\documentclass{article}
\usepackage{xcolor} 
\usepackage{amsmath,amsthm,amsmath,amsfonts,amssymb,geometry}
\usepackage{graphicx} 
\usepackage[authoryear]{natbib}
\usepackage{bbm,mathrsfs,subfigure,titlesec,lipsum,enumerate,bm,here,multicol,authblk,comment}
\usepackage{microtype}

\makeatother

\def\R{\mathbb{R}}
\def\N{\mathbb{N}}

\def\E{\mathrm{E}}
\def\var{{\mathrm{Var}}}

\numberwithin{equation}{section}
\theoremstyle{plain}
\newtheorem{theorem}{Theorem}[section]
\newtheorem{lemma}[theorem]{Lemma}

\newtheorem{corollary}[theorem]{Corollary}
\theoremstyle{definition}

\newtheorem{assumption}{Assumption}
\newtheorem{remark}{Remark}
\newtheorem{example}{Example}

\title{High-dimensional linear regression inference\\ via $\ell^2$ weak convergence}
\author[1]{Kou Fujimori}
\author[2]{Koji Tsukuda}
\affil[1]{Faculty of Economics and Law, Shinshu University.}
\affil[2]{Faculty of Mathematics, Kyushu University.}
\date{}

\begin{document}
\maketitle

\begin{abstract}
We prove weak convergence in a separable Hilbert space for estimators of high-dimensional regression coefficients, which yields asymptotic normality and enables direct use of standard asymptotic tools such as the continuous mapping theorem.
The approach permits diverging sparsity with a growing number of small nonzero coefficients, while requiring that only finitely many have moderate magnitude.
As applications, we develop a test for finitely many linear hypotheses and, via a Scheff\'{e}-type approach, simultaneous inference for infinitely many linear hypotheses, yielding both a global test and simultaneous confidence bands for the regression function.
The limiting distributions are given by weighted sums of independent chi-squared variables, and plug-in critical values achieve asymptotically correct size.
\end{abstract}

\noindent
\begin{list}{}{
\setlength{\leftmargin}{-5pt}
\setlength{\rightmargin}{-5pt}
}
\item[] \textls[-25]{\textbf{Keywords:} high-dimensional inference; linear regression; simultaneous inference; sparsity; weak convergence.}\\
\end{list}

\section{Introduction}\label{sec:intro}

We establish asymptotic normality for estimators of high-dimensional linear regression coefficients through weak convergence in the Hilbert space $\ell^2$ and provide a basis for global testing and Scheff\'{e}-type simultaneous inference by applying standard tools for continuous functionals.
Unlike existing results based on fixed sparsity or low-dimensional functionals, our framework allows the number of nonzero coefficients to diverge. 
We assume that only finitely many coefficients have moderate magnitude and that the total squared magnitude of the remaining weak coefficients is sufficiently small.
For clarity of exposition, we focus on linear regression with independent and identically distributed (i.i.d.) sub-Gaussian covariates and errors and defer all proofs to the final section.

A wide variety of methods for high-dimensional sparse estimation have been proposed, including the least absolute shrinkage and selection operator (lasso)~\citep{tibshirani1996}, the adaptive lasso~\citep{zou2006}, the Dantzig selector~\citep{candes2007dantzig}, and losses with nonconvex penalties such as the smoothly clipped absolute deviation and the minimax concave penalty~\citep{fan2001,zhang2010}.
The oracle property, introduced by \citet{fan2001}, combines asymptotically exact support recovery with asymptotic normality and has been studied.
Moreover, foundational results include oracle inequalities and support-recovery guarantees~\citep{bickel2009simultaneous,wainwright2009}.
On the inferential side, post-selection procedures have been analyzed~\citep{belloni2013least}; separately, two-step estimators for stochastic process models have been developed~\citep{fujimori2026two}; and de-biasing methods yield asymptotically normal estimators for low-dimensional functionals~\citep{zhangzhang2014,vandegeer2014,javanmard2014}.

Accordingly, an extensive literature has developed non-asymptotic, high-dimensional central limit theorems and Gaussian approximations under metrics such as the Kolmogorov and Wasserstein distances.
Representative results include Gaussian approximations for maxima of sums of centered independent random vectors~\citep{chernozhukov2013gaussian,chernozhukov2017central}, Stein kernel–based approximations accommodating exchangeable pairs and certain nonlinear statistics~\citep{fang2021high}, rate refinements~\citep{koike2023high}, and extensions to convex polytopes and degenerate cases~\citep{fang2023high}.

Existing debiased estimators and high-dimensional Gaussian approximation methods primarily yield coordinatewise or finite-dimensional inference, typically via non-asymptotic error bounds. 
They do not, however, establish weak convergence of the estimation error in $\ell^2$ or provide a unified basis for inference on its continuous functionals.
In contrast to metric-specific non-asymptotic bounds, weak convergence in $\ell^2$ allows standard asymptotic tools such as the continuous mapping theorem to be applied directly.
Coordinatewise normality and normality of fixed-dimensional linear functionals follow immediately, whereas coordinatewise normality alone does not capture the joint dependence structure required for simultaneous or global inference.
By contrast, weak convergence in $\ell^2$ characterizes the joint Gaussian limit through its covariance operator and supports inference based on the $\ell^2$ norm or multiple linear hypotheses.
As applications, we construct tests of linear hypotheses and compare them with projection-based procedures for dense nuisance parameters~\citep{zhu2018linear} and hypothesis-adaptive tests designed for high power~\citep{zhang2025novel}. 
Our framework permits the nominal sparsity level to diverge by explicitly controlling the cumulative effect of the weak-signal tail.
We also develop a Scheff\'{e}-type procedure for simultaneous testing over an infinite continuum of linear contrasts, which yields simultaneous confidence bands for the regression function.

Throughout the paper, we use the following notation.
We write $\mathbb{N} = \{1,2,\ldots\}$ and $[k] = \{1,\ldots,k\}$ for $k \in \mathbb{N}$.
For $q\in \mathbb{N} \cup \{\infty\}$, write $\|\boldsymbol{v}\|_q$ for the $\ell^q$ norm: 
for $\boldsymbol{v} = (v_1,\ldots,v_p)^\top \in \mathbb{R}^p$ with $p \in \mathbb{N}$,
$\|\boldsymbol{v}\|_q = (\sum_{j=1}^p |v_j|^q )^{1/q}$ $(q<\infty)$, $\|\boldsymbol{v}\|_\infty = \max_{1\le j\le p} |v_j|$.
For an infinite sequence $\boldsymbol{v}=(v_j)_{j\ge1}\in\ell^q$, the same formulae apply.
For a bounded linear operator $\boldsymbol{A}: \ell^2 \to \ell^2$, or equivalently a matrix $\boldsymbol{A}$ in finite dimensions, let
$\|\boldsymbol{A}\|_{\mathrm{op}} = \sup_{\|\boldsymbol{x}\|_2=1}\|\boldsymbol{A}\boldsymbol{x}\|_2$
denote the $\ell^2$‑induced operator norm.
For a vector $\boldsymbol{v}=(v_j)_{j\ge 1}$ and an index set $T\subset\mathbb{N}$ with $|T|<\infty$, write $\boldsymbol{v}_T = (v_j)_{j\in T}$, which is viewed as an element of $\mathbb{R}^{|T|}$. 
In the finite‑dimensional case $\boldsymbol{v}\in\mathbb{R}^p$, we restrict $T\subset[p]$; for sequences in $\ell^2$, we allow $T\subset \mathbb{N}$.
Similarly, for a $p \times p$ matrix $\boldsymbol{A} = (A_{ij})_{i,j \in [p]}$ and index sets $T, T' \subset [p]$, denote $(A_{ij})_{i \in T, j \in T'}$ by $\boldsymbol{A}_{T, T'}$ or $(\boldsymbol{A})_{T, T'}$.
For a random variable $X$, we write the sub-Gaussian norm of $X$, if it exists, as $\|X\|_{\psi_2}$, that is, $\|X\|_{\psi_2} = \inf\{ C>0: \E [\exp ({X^2}/{C^2}) ]\leq 2\}$.
For a $d$-dimensional random vector $\boldsymbol{X}$, the sub-Gaussian norm of $\boldsymbol{X}$ is 
$\|\boldsymbol{X}\|_{\psi_2}
= \sup_{\|\boldsymbol{u}\|_2 =1 } \| \boldsymbol{u}^\top \boldsymbol{X} \|_{\psi_2}$.
For $d \in \mathbb{N}$, $\boldsymbol{\mu} \in \mathbb{R}^d$, and $\boldsymbol{V} \in \mathbb{R}^{d \times d}$ positive semi‑definite, we write $\mathrm{N}_d(\boldsymbol{\mu}, \boldsymbol{V})$ for the $d$-dimensional normal distribution with mean $\boldsymbol{\mu}$ and covariance $\boldsymbol{V}$, with the subscript omitted when $d=1$.
For random elements $X$ and $Y$ and a sequence $\{X_n\}_{n \in \mathbb{N}}$ of random elements, 
we write $X =^d Y$ if $X$ and $Y$ have identical distributions, and we write $X_n \to^d X$ and $X_n \to^p X$ for convergence in distribution and in probability as $n \to \infty$, respectively.

\section{Model setup and main result}\label{sec:model}

For each $n \in \mathbb{N}$, we work with the triangular array $\{(Y_{i,n}, \boldsymbol{Z}_{i,n}^{\top} )\}_{i \in [n]}$ satisfying
\[
Y_{i,n} = \boldsymbol{\theta}_n^\top \boldsymbol{Z}_{i,n} + \epsilon_{i,n}, \quad i \in [n],
\]
where $\boldsymbol{\theta}_n \in \mathbb{R}^{p_n}$ and $\boldsymbol{Z}_{i,n} =(Z_{i1,n}, \ldots, Z_{i p_n,n})^\top \in \mathbb{R}^{p_n}$ for $i \in [n]$.
Assume that 
\[ (\boldsymbol{Z}_{1,n}^{\top}, \epsilon_{1,n})^{\top}, \ldots, (\boldsymbol{Z}_{n,n}^{\top}, \epsilon_{n,n})^{\top} \]
are i.i.d.\ with $\E(\boldsymbol{Z}_{1,n}) = \boldsymbol{0}$, $\E(\epsilon_{1,n}) = 0$, $\var(\epsilon_{1,n}) = \sigma^2 > 0$, and $\boldsymbol{Z}_{1,n}$ is independent of $\epsilon_{1,n}$.
Let $\boldsymbol{\theta}_{0,n} = (\theta_{01,n}, \ldots, \theta_{0p_n,n})^\top$ denote the true parameter vector.
Define the index sets
\begin{align*}
&T_\gamma := \{ j \in [p_n] : |\theta_{0j,n}| > \gamma \},\\
&T_{0,n} := \{ j \in [p_n] : \theta_{0j,n} \neq 0 \},\\
&T_{W,n} := \{ j \in [p_n] : \theta_{\min,n} \le |\theta_{0j,n}| \le \eta_n \},
\end{align*}
where
$\gamma > 0$ is a fixed constant independent of $n$ and $p_n$, 
\[ \theta_{\min,n} := \min_{j \in T_{0,n}} |\theta_{0j,n}|,\]
and $\eta_n$ is a positive sequence depending on $n$.
We write $s_{0,n} := |T_{0,n}|$ for the support size, which may grow with $n$, and $s_\gamma := |T_\gamma|$ for the number of coefficients larger than $\gamma$ in magnitude.
Unless otherwise stated, throughout this paper, all limits are taken as $n \to \infty$, allowing $p_n$ to diverge.

\begin{assumption}\label{assump:sparsity}
\begin{enumerate}
\item[(i)]
$T_\gamma$ is fixed and independent of $n$ and $p_n$.
\item[(ii)]
$T_{0,n} \setminus T_\gamma =  T_{W,n}$.
\item[(iii)]
There exists $\theta^* >0$, independent of $n$ and $p_n$, such that 
$\| \boldsymbol{\theta}_{0,n} \|_{\infty} \leq \theta^*$.
\item[(iv)]
$\sigma^2$ is fixed and independent of $n$ and $p_n$.
\end{enumerate}
\end{assumption}


\begin{remark}
By Assumption~\ref{assump:sparsity}-(i), the number $s_\gamma$ of non-negligible coefficients remains fixed.
Assumption~\ref{assump:sparsity}-(ii) allows for a diverging number of nonzero coefficients outside $T_\gamma$, while their total squared magnitude is controlled by Assumption~\ref{assump:energy}.
\end{remark}

Let $\hat{\boldsymbol{\theta}}_n = (\hat\theta_{n1},\ldots,\hat\theta_{n p_n})^\top$ be a sparse estimator for $\boldsymbol{\theta}_{n}$ such that 
\[ \Pr (\|\hat{\boldsymbol{\theta}}_n - \boldsymbol{\theta}_{0,n} \|_\infty
> r_n )\to 0\]
for some rate $r_n \to 0$.
Given a threshold sequence $\{ \tau_n \}_{n \in \mathbb{N}}$, define the sequence $\{ \hat{T}_n \}_{n \in \mathbb{N}}$ of an estimator for $T_\gamma$ by
\[
\hat{T}_n  := \{j \in [p_n] : |\hat{\theta}_{nj}| > \tau_n\}.
\]

\begin{remark}
As discussed in \cite{bickel2009simultaneous}, under suitable conditions,
\[\|\hat{\boldsymbol{\theta}}_n - \boldsymbol{\theta}_{0,n} \|_\infty \leq \|\hat{\boldsymbol{\theta}}_n - \boldsymbol{\theta}_{0,n}\|_2 \le C n^{-1/2} (s_{0,n}\log{p_n})^{1/2}
\]
for some $C>0$ with probability tending to one. 
In such cases, one may take $r_n \asymp n^{-1/2} (s_{0,n}\log p_n)^{1/2}$.
\end{remark}

\begin{assumption}\label{assump:selcons}
$r_n + \eta_n < \tau_n < \gamma - r_n$.
\end{assumption}

\begin{lemma} \label{lem:selection}
Let Assumptions~\ref{assump:sparsity}--\ref{assump:selcons} hold.
Then, 
\[ \Pr(\hat{T}_n = T_\gamma) \to 1.\]
\end{lemma}

The post-selection ordinary least squares estimator $\tilde{\boldsymbol{\theta}}_n = (\tilde{\theta}_{n1},\ldots,\tilde{\theta}_{n p_n})^{\top}$ is defined as the solution to the estimating equation
\[
n^{-1} \sum_{i=1}^n \boldsymbol{Z}_{i,n \hat{T}_n}(Y_{i,n} - \boldsymbol{\theta}_{\hat{T}_n}^\top \boldsymbol{Z}_{i,n \hat{T}_n}) = \boldsymbol{0},\quad
\boldsymbol{\theta}_{\hat{T}_n^c} = \boldsymbol{0}.
\]
Consider the $\ell^2$-valued random sequence 
$\{\boldsymbol{\mathcal{R}}_n\}_{n \in \mathbb{N}}$ defined by 
\[
\langle \boldsymbol{e}_j, \boldsymbol{\mathcal{R}}_n\rangle
= \begin{cases}
n^{1/2} (\tilde{\theta}_{nj} - \theta_{0j,n} ) & 
(j \in [p_n]) \\
0 & (j \in \mathbb{N} \setminus [p_n])
\end{cases},
\] 
where $\{\boldsymbol{e}_j\}_{j \in \mathbb{N}}$ is the canonical basis of $\mathbb{R}^\infty$. 
Define
\[
\hat{\boldsymbol{J}}_n := n^{-1}\sum_{i=1}^n \boldsymbol{Z}_{i,n} \boldsymbol{Z}_{i,n}^\top,\quad
\boldsymbol{J}_n := \E[\boldsymbol{Z}_{1,n} \boldsymbol{Z}_{1,n}^\top].
\]

\begin{assumption}\label{assump:regularity}
\begin{enumerate}
\item[(i)]
There exists a $s_\gamma \times s_\gamma$ positive definite matrix 
\[ \boldsymbol{J}_{T_\gamma, T_\gamma} 
= \lim ( \boldsymbol{J}_n )_{T_\gamma, T_\gamma}. \]
\item[(ii)]
There exists $\lambda_*$, independent of $n$ and $p_n$, such that 
\[0 < \lambda_{*} \leq  \Lambda_{\min}( ( \boldsymbol{J}_n )_{T_\gamma, T_\gamma}) , \]
where $\Lambda_{\min} ( ( \boldsymbol{J}_n )_{T_\gamma, T_\gamma})$ is the smallest eigenvalue of $( \boldsymbol{J}_n )_{T_\gamma, T_\gamma}$.
\item[(iii)]
$\sup_{n \in \mathbb{N}} \max_{j \in [p_n]} \| Z_{1 j,n} \|_{\psi_2} < \infty$.
\item[(iv)]
$\sup_{n \in \mathbb{N}} \| \epsilon_{1 ,n} \|_{\psi_2} < \infty$.
\end{enumerate}
\end{assumption}

\begin{lemma}\label{lem:matrix-approx}
Let Assumptions~\ref{assump:sparsity}-(i) and \ref{assump:regularity}-(ii)(iii) hold.
Then, 
\[ \| ( \hat{\boldsymbol{J}}_n )_{T_\gamma, T_\gamma} - ( \boldsymbol{J}_n )_{ T_\gamma, T_\gamma} \|_{\mathrm{op}}
= o_p(1). \]
\end{lemma}

\begin{assumption}\label{assump:energy}
There exists $\kappa \in [0, 1/2]$, independent of $n$ and $p_n$, such that
\begin{equation}\label{eq:kappa}
\sup_{n \in \mathbb{N}} s_{0,n}^{- \kappa} \| \boldsymbol{Z}_{1,n T_{0,n}} \|_{\psi_2}  < \infty
\end{equation} 
and $n s_{0,n}^{1 + 2\kappa} \eta_n^2 \to 0$.
\end{assumption}

\begin{remark}
Assumption~\ref{assump:energy} controls the cumulative effect of the coefficients outside $T_\gamma$, taking account of the dependence structure of the corresponding covariates.
Together with Assumption~\ref{assump:sparsity}, it yields $(\boldsymbol{\theta}_{0,n}^{\top},0,0,\ldots)^{\top} \in \ell^2$ uniformly in $n$.
Indeed, since $s_\gamma$ is fixed,
\[ \|\boldsymbol{\theta}_{0,n}\|_2^2
\leq s_\gamma {\theta^*}^2 + s_{0,n} \eta_n^2 =O(1).\]
Assumption~\ref{assump:energy} restricts the growth of the support size, the magnitude of the weak coefficients, and the dependence structure of the corresponding covariates.
\end{remark}

\begin{remark}
Let us illustrate the two extreme values of $\kappa$ in Assumption~\ref{assump:energy}.
Suppose Assumption~\ref{assump:regularity}-(iii).
Then, 
\[ \| \boldsymbol{Z}_{1,n T_{0,n}} \|_{\psi_2} \leq s_{0,n}^{1/2} \max_{j \in T_{0,n}} \| Z_{1j,n} \|_{\psi_2}, \] 
and hence 
\[ \sup_{n \in \mathbb{N}} s_{0,n}^{-1/2} \| \boldsymbol{Z}_{1,n T_{0,n}} \|_{\psi_2} < \infty. \] 
Thus, \eqref{eq:kappa} always holds with $\kappa=1/2$.
Moreover, if the components of $\boldsymbol{Z}_{1,n T_{0,n}}$ are independent, then 
\[
\sup_{n \in \mathbb{N}} \| \boldsymbol{Z}_{1,n T_{0,n}} \|_{\psi_2} 
\leq 
C \sup_{n \in \mathbb{N}} \max_{j \in [p_n]} \| Z_{1 j,n} \|_{\psi_2} < \infty
\]
for some universal constant $C>0$, which implies that \eqref{eq:kappa} holds with $\kappa=0$.
\end{remark}

\begin{theorem}[Asymptotic normality] \label{thm:ell2_weak}
Let Assumptions~\ref{assump:sparsity}--\ref{assump:energy} hold.
Then, 
\[\boldsymbol{\mathcal{R}}_n \to^d \boldsymbol{\mathcal{R}} \quad \text{in} \quad \ell^2, \]
where $\boldsymbol{\mathcal{R}}$ is a centered Gaussian field satisfying
\[ \langle \boldsymbol{u}, \boldsymbol{\mathcal{R}} \rangle \sim \mathrm{N}(0, \sigma^2\boldsymbol{u}_{T_\gamma}^\top 
( \boldsymbol{J}_{T_\gamma, T_\gamma} )^{-1} \boldsymbol{u}_{T_\gamma}) \]
for every $\boldsymbol{u} \in \ell^2$.
\end{theorem}

\begin{remark}
Recall that $\boldsymbol{\mathcal{R}}_n \to^d \boldsymbol{\mathcal{R}}$ in $\ell^2$ means that $\E[f(\boldsymbol{\mathcal{R}}_n)] \to \E[f(\boldsymbol{\mathcal{R}})]$ for every bounded, continuous function $f:\ell^2\to\mathbb{R}$.
This notion is stronger than convergence of linear functionals and identifies the limiting law as a Borel probability measure on $\ell^2$.
\end{remark}

\begin{remark}
The oracle property includes asymptotic normality on the recovered support.
Theorem~\ref{thm:ell2_weak} may be viewed as a result in $\ell^2$ of its asymptotic normality component.
\end{remark}

\begin{example}
Let $\Delta > 0$ and $0<\gamma<1$, and consider $\boldsymbol{\theta}_{0,n}$ defined by
\[ \theta_{0j,n} = 1_{\{1 \leq j \leq s_\gamma \}} + n^{-(1+\Delta)/2} 1_{\{s_\gamma < j \leq s_{0,n} \}} \quad (j \in [p_n]). \]
Suppose that $s_{0,n} = o(n^{\Delta})$ and $\liminf n^{1/2} r_n >0$.
Then $r_n + \eta_n \sim r_n$, so Assumption~\ref{assump:selcons} holds for $\tau_n$ satisfying $\tau_n \to 0$ and $r_n/\tau_n \to 0$.
Moreover, suppose that the components of $\boldsymbol{Z}_{1,n T_{0,n}}$ are independent. Then, under Assumption~\ref{assump:regularity}-(iii), \eqref{eq:kappa} holds with $\kappa=0$, and Assumption~\ref{assump:energy} follows since $n s_{0,n} \eta_n^2 = s_{0,n} n^{-\Delta} \to 0$.
Consequently, if the remaining assumptions of Theorem~\ref{thm:ell2_weak} hold, then $\boldsymbol{\mathcal{R}}_n \to^d \boldsymbol{\mathcal{R}}$ in $\ell^2$.
\end{example}

\begin{example}\label{example:2}
Let $\Delta$, $\gamma$, $\boldsymbol{\theta}_{0,n}$, and $\boldsymbol{Z}_{1,n T_{0,n}}$ be as in the previous example, so \eqref{eq:kappa} holds with $\kappa=0$.
If $s_{0,n} \asymp n^\Delta$, then $(\boldsymbol{\theta}_{0,n}^\top, 0 , \ldots)^\top \in \ell^2$ but Assumption~\ref{assump:energy} fails because 
\[ \liminf ( n s_{0,n}^{1 + 2\kappa} \eta_n^2 ) \geq 
 \liminf ( n s_{0,n} \eta_n^2) > 0 \]
for any $\kappa \in [0,1/2]$.
Assume for contradiction that $\boldsymbol{\mathcal{R}}_n \to^d \boldsymbol{\mathcal{R}}$ in $\ell^2$.
Define $g:\ell^2 \to \mathbb{R}$ by $g(\cdot)= \sum_{j>s_\gamma}\langle \boldsymbol{e}_j, \cdot \rangle^2$.
Then, by the continuous mapping theorem, $g(\boldsymbol{\mathcal{R}}_n)  \to^d g(\boldsymbol{\mathcal{R}})$.
Since $g(\boldsymbol{\mathcal{R}}) = 0$ a.s., it follows that $g(\boldsymbol{\mathcal{R}}_n) \to^d 0$.
On the event $\{\hat{T}_n=T_\gamma\}$ we have $\tilde\theta_{nj}=0$ for $j>s_\gamma$,
and hence $\langle \boldsymbol{e}_j,\boldsymbol{\mathcal{R}}_n\rangle^2 = n^{-\Delta}$
for $s_\gamma<j\le s_{0,n}$.
Therefore, for any $c > 0$,
\begin{align*}
&\Pr \left( g(\boldsymbol{\mathcal{R}}_n) > c \right)  \\
&\geq 
\Pr \left( \sum_{j = s_\gamma + 1}^{s_{0,n} } \langle \boldsymbol{e}_j , \boldsymbol{\mathcal{R}}_n \rangle^2 > c ,  \  \hat{T}_n = T_\gamma  \right) \\
&= \Pr \left(  (s_{0,n} - s_\gamma) n^{-\Delta} > c, \  \hat{T}_n = T_\gamma \right).
\end{align*}
If $c < \liminf (s_{0,n} / n^{\Delta})$, then Lemma~\ref{lem:selection} yields
\[
\liminf \Pr ( g(\boldsymbol{\mathcal{R}}_n) > c )
\geq \lim \Pr(\hat{T}_n = T_\gamma) = 1,
\]
which contradicts $g(\boldsymbol{\mathcal{R}}_n)  \to^d 0$.
\end{example}

A weak convergence result allowing for a non-negligible mean shift without Assumption~\ref{assump:energy} is presented in Section~\ref{sec:local}.

\begin{remark}
\citet{fujimori2026two} analyze asymptotic normality of estimators for high‑dimensional parameters via weak convergence in $\ell^2$, including settings with a possibly infinite‑dimensional nuisance parameter and dependent observations. 
Their analysis is confined to the fixed-sparsity regime, which materially simplifies the arguments.
In contrast, we allow $s_{0,n} \to \infty$ while keeping $s_\gamma$ fixed.
\end{remark}

\section{Testing linear hypotheses}\label{sec:appli1}

As a direct application of Theorem~\ref{thm:ell2_weak}, we consider testing linear hypotheses for $\boldsymbol{\theta}_{n}$.
Let $q (\leq p_{n})$ be an integer, $\boldsymbol{A}_{n} \in \mathbb{R}^{q \times p_{n}}$ a given matrix, and $\boldsymbol{b} \in \mathbb{R}^q$ a given vector.
Let 
\[ \boldsymbol{\Sigma}_{A, n} 
= (\boldsymbol{A}_{n})_{[q], T_\gamma} \boldsymbol{J}_{T_\gamma, T_\gamma}^{-1} (\boldsymbol{A}_{n})_{[q], T_\gamma}^\top .\]
To test
$H_0: \boldsymbol{A}_{n} \boldsymbol{\theta}_{n} = \boldsymbol{b}$ against $H_1: \boldsymbol{A}_{n} \boldsymbol{\theta}_{n} \neq \boldsymbol{b}$, consider
\[
W_n := n \|\boldsymbol{A}_{n} \tilde{\boldsymbol{\theta}}_n - \boldsymbol{b}\|_2^2.
\]
Furthermore, for $j \in [q]$, denote by $\boldsymbol{a}_{j,n}$ the $j$-th row of $\boldsymbol{A}_n$ and by $\tilde{\boldsymbol{a}}_{j,n}$ its zero-padded version $\tilde{\boldsymbol{a}}_{j,n} = (\boldsymbol{a}_{j,n}^\top,0,0,\ldots)^\top$.

\begin{assumption}\label{assump:onq}
\begin{enumerate}
\item[(i)]
$q$ is fixed and independent of $n$ and $p_n$.
\item[(ii)]
$(\boldsymbol{A}_n)_{[q], T_\gamma}$ is a fixed nonzero matrix, independent of $n$ and $p_n$.
\item[(iii)]
For all $j\in[q]$, there exists $\boldsymbol{a}_{j} = \lim \tilde{\boldsymbol{a}}_{j,n}$ in $\ell^2$.
\end{enumerate}
\end{assumption}

\begin{remark}
Assumption~\ref{assump:onq}-(iii) holds if, for all $j\in[q]$, there exists $\boldsymbol{a}_{j} \in \ell^2$ such that $\langle \boldsymbol{x}, \tilde{\boldsymbol{a}}_{j,n} \rangle \to \langle \boldsymbol{x}, \boldsymbol{a}_j \rangle$ for every $\boldsymbol{x} \in \ell^2$ and $\limsup \| \tilde{\boldsymbol{a}}_{j,n} \|_2 \leq \| \boldsymbol{a}_j \|_2$.
\end{remark}

Under Assumption~\ref{assump:onq}, $\boldsymbol{\Sigma}_{A,n}$ is independent of $n$ and $p_n$.
Denote it by $\boldsymbol{\Sigma}_A$. 
Let $r$ be its rank and $\Lambda_1 \geq \ldots \geq \Lambda_r$ its nonzero eigenvalues.
Since $s_\gamma$ and $q$ are fixed, we have $r \leq \min(s_\gamma, q)$, so $r$ is fixed.

\begin{corollary}\label{cor:test asymptotic distribution}
Let Assumptions~\ref{assump:sparsity}--\ref{assump:onq} hold.
Under $H_0$,
\[
W_n \to^d \sigma^2 \sum_{j=1}^r \Lambda_j \chi_{1,j}^2 ,
\]
where $\chi_{1,1}^2,\ldots, \chi_{1,r}^2$ are independent and $\chi^2(1)$-distributed random variables.
\end{corollary}

We consider simple estimators for $\sigma^2$ and $\Lambda_1,\ldots,\Lambda_r$.
For $\sigma^2$, define
\[
\hat{\sigma}_n^2
= n^{-1} \sum_{i=1}^n \left( Y_{i,n} - \tilde{\boldsymbol{\theta}}_n^\top \boldsymbol{Z}_{i,n} \right)^2.
\]
For $\Lambda_j$ ($j\in[r]$), define $\hat{\Lambda}_{n,j}$ as the $j$-th largest eigenvalue of
\[ 
\hat{\boldsymbol{\Sigma}}_{A,n} 
= (\boldsymbol{A}_{n})_{[q], \hat{T}_n} \{(\hat{\boldsymbol{J}}_n)_{ \hat{T}_n, \hat{T}_n}\}^{-1} (\boldsymbol{A}_n)
_{[q], \hat{T}_n}^\top . \]
Let $\hat{r}_n$ be the rank of $\hat{\boldsymbol{\Sigma}}_{A, n}$.

\begin{lemma}\label{lem:sigma_estimation}
Let Assumptions~\ref{assump:sparsity}--\ref{assump:energy} hold.
Then, 
\[ \hat{\sigma}_n^2 \to^p \sigma^2.\]
\end{lemma}

\begin{lemma}\label{lem:lambda_estimation}
Let Assumptions~\ref{assump:sparsity}--\ref{assump:onq} hold.
Then, 
\[ \max_{j \in [r]}|\hat{\Lambda}_{n, j} - \Lambda_j|1_{\{\hat{T}_n = T_\gamma\}} \to^p 0.\]
\end{lemma}

Let $c_\alpha$, $\hat{c}_\alpha$ denote the upper $\alpha$-quantiles of the distributions of
\[ \sum_{j=1}^r \Lambda_j \chi_{1,j}^2, \quad  \sum_{j=1}^{\hat{r}_n} \hat{\Lambda}_{n,j} \chi_{1,j}^2,\]
respectively.
Thus $\hat{c}_\alpha$ serves as a natural estimator for $c_\alpha$.
We reject $H_0$ if $W_n/\hat{\sigma}_n^2  > \hat{c}_\alpha$.
Equivalently, the test function is
\[ \varphi_n = 1_{\{W_n/\hat{\sigma}_n^2  > \hat{c}_\alpha\}}.\]

\begin{theorem} \label{thm:test_rule1}
Let Assumptions~\ref{assump:sparsity}--\ref{assump:onq} hold.
Under $H_0$, 
\[ \Pr(\varphi_n = 1) \to \alpha. \]
\end{theorem}

\begin{remark}\label{rem:critical}
The critical value $\hat{c}_\alpha$ is obtained by Monte Carlo simulation, with a moment-matching approximation as a faster alternative.
\end{remark}

\begin{remark}
\citet{zhu2018linear} and \citet{zhang2025novel} conduct statistical inference with dense $\boldsymbol{\theta}_{n}$ by imposing sparsity on orthogonalized loading directions obtained after suitable projection.
In contrast, we impose no sparsity on $\boldsymbol{A}_{n}$ and instead proceed under stronger assumptions on $\boldsymbol{\theta}_{n}$, thereby situating our results in a different inferential setting.
\end{remark}

\section{Testing infinitely many linear hypotheses}\label{sec:appli2}

Unlike classical multiple testing procedures for finitely many hypotheses, our approach achieves asymptotically exact family-wise error control over an infinite continuum of linear contrasts.
This result follows from weak convergence in $\ell^2$, which yields a joint limit law adequate for global and simultaneous inference.
In particular, as an application of Theorem~\ref{thm:ell2_weak}, we consider a Wald--Scheff\'{e}-type formulation based on an orthogonal decomposition.

Let
\[
\mathscr{A}_{n}
= \left\{ \boldsymbol{a} \in \mathbb{R}^{p_{n}} : \operatorname{supp}(\boldsymbol{a}) \subset T_\gamma , \|\boldsymbol{a}\|_2 = 1 \right\} \]
and let $c(\boldsymbol{a})$ be a pre-specified contrast value for each $\boldsymbol{a} \in \mathscr{A}_{n}$.
Under Assumption~\ref{assump:sparsity}, $\mathscr{A}_n$ and its elements depend on $n$ only through zero padding outside $T_\gamma$. 
We therefore suppress this dependence and write $\mathscr{A}$ and $\boldsymbol{a} (\in \mathscr{A})$ for simplicity.
Define the affine hyperplane 
\[ H(\boldsymbol{a}) := \{ \boldsymbol{\theta} \in \mathbb{R}^{p_{n}} : \boldsymbol{a}^\top \boldsymbol{\theta} = c(\boldsymbol{a}) \}. \]
We consider the family 
\[ \mathscr{F} = \{ H_{0,\boldsymbol{a}} : \boldsymbol{a} \in \mathscr{A}_{n} \},\]
where $H_{0,\boldsymbol{a}}: \boldsymbol{a}^\top \boldsymbol{\theta}_{n} = c(\boldsymbol{a})$.
A null hypothesis $H_{0,\boldsymbol{a}}$ can be rewritten as $\boldsymbol{\theta}_{n} = \Pi_{H(\boldsymbol{a})} \boldsymbol{\theta}_{n}$ with  $\Pi_{H(\boldsymbol{a})}$ the orthogonal projection onto $H(\boldsymbol{a})$ with respect to the standard Euclidean inner product.
To test $\mathscr{F}$, consider
\[
W^{(\mathscr{A})}_n := n \sup_{\boldsymbol{a}\in\mathscr{A}} \| \tilde{\boldsymbol{\theta}}_n -  \Pi_{H(\boldsymbol{a})} \tilde{\boldsymbol{\theta}}_n \|_2^2 .
\]

\begin{lemma}\label{lem:proj}
Suppose that
$\boldsymbol{a}^\top \boldsymbol{\theta}_{0,n} = c(\boldsymbol{a})$ for all $\boldsymbol{a} \in \mathscr{A}$.
Then
\[ W^{(\mathscr{A})}_n = n\| \Pi_{T_\gamma} ( \tilde{\boldsymbol{\theta}}_n - \boldsymbol{\theta}_{0,n}) \|_2^2,\]
where $\Pi_{T_\gamma}:\mathbb{R}^{p_{n}} \to \mathbb{R}^{p_{n}}$ denotes the projection onto $ \mathrm{span}\{\boldsymbol{e}_j ; j \in T_\gamma \}$.
\end{lemma}

\begin{remark}
The reduction of the supremum over linear contrasts to a quadratic form relies on the Hilbert space structure of $\ell^2$ via the Cauchy--Schwarz inequality.
\end{remark}

\begin{corollary}\label{cor:test2_asydist}
Let Assumptions~\ref{assump:sparsity}--\ref{assump:energy} hold.
Under $\mathscr{F}$,
\[
W^{(\mathscr{A})}_n  \to^d \sigma^2 \sum_{j=1}^{s_\gamma} \lambda_j \chi_{1,j}^2 ,
\]
where $\lambda_1 \geq \ldots \geq \lambda_{s_\gamma}$ are the eigenvalues of $( \boldsymbol{J}_{T_\gamma, T_\gamma} )^{-1}$, and $\chi_{1,1}^2,\ldots,\chi_{1,s_\gamma}^2$ are independent and $\chi^2(1)$-distributed random variables.
\end{corollary}

The contrast value $c(\boldsymbol{a})$ is fixed a priori to define a meaningful family-wise null hypothesis.
A natural choice is $c(\boldsymbol{a}) = \boldsymbol{a}^\top \boldsymbol{\theta}_{\mathrm{NULL}}$ for a given $\boldsymbol{\theta}_{\mathrm{NULL}}$, in which case the corresponding omnibus test can be implemented via
\[\hat{W}^{(\mathscr{A})}_n = n\| \Pi_{\hat{T}_n} ( \tilde{\boldsymbol{\theta}}_n - \boldsymbol{\theta}_{\mathrm{NULL}}) \|_2^2 \]
instead of $W^{(\mathscr{A})}_n = n\| \Pi_{{T}_\gamma} ( \tilde{\boldsymbol{\theta}}_n - \boldsymbol{\theta}_{\mathrm{NULL}}) \|_2^2$.
Hereafter we consider $c(\boldsymbol{a}) = \boldsymbol{a}^\top \boldsymbol{\theta}_{\mathrm{NULL}}$ $(\boldsymbol{a} \in \mathscr{A})$.
To estimate $\sigma^2$, we use $\hat{\sigma}_n^2$ introduced in Section~\ref{sec:appli1}.
For $\lambda_1,\ldots,\lambda_{s_\gamma}$, we consider the simple estimator $\hat{\lambda}_{n,j}$ as the $j$-th largest eigenvalue of $\{ ( \hat{\boldsymbol{J}}_n )_{\hat{T}_n, \hat{T}_n} \}^{-1}$ for $j\in [\hat{s}_\gamma]$, where $\hat{s}_\gamma = | \hat{T}_n |$.
Moreover, let $q_{\alpha}$, $\hat{q}_\alpha$ denote the upper $\alpha$-quantiles of the distributions of
\[ \sum_{j=1}^{s_\gamma} \lambda_j \chi_{1,j}^2, \quad  \sum_{j=1}^{\hat{s}_\gamma} \hat{\lambda}_{n,j} \chi_{1,j}^2,\]
respectively.
We use the test function 
\[ \varphi^{(\mathscr{A})}_n = 1_{\{\hat{W}^{(\mathscr{A})}_n/\hat{\sigma}_n^2  > \hat{q}_\alpha\}}.\]

\begin{theorem}\label{thm:test_rule2}
Let Assumptions~\ref{assump:sparsity}--\ref{assump:energy} hold.
Under $\mathscr{F}$, 
\[ \Pr(\varphi^{(\mathscr{A})}_n = 1) \to \alpha. \]
\end{theorem}

\begin{remark}
The computation of the critical value $\hat{q}_\alpha$ follows the same procedure as in Remark~\ref{rem:critical}, and no additional computational difficulties arise from the simultaneous nature of the test.
\end{remark}

Because of Lemma~\ref{lem:selection}, post hoc contrasts supported on the selected components can be handled within the Scheff\'{e}-type framework.
This point is illustrated in the following example.

\begin{example}
Let $\boldsymbol{\theta}_{\mathrm{NULL}}$ be given.
We consider the family $\mathscr{F}$ of null hypotheses
\[
H_{0,\boldsymbol{a}}: \boldsymbol{a}^\top \boldsymbol{\theta}_{n} = \boldsymbol{a}^\top \boldsymbol{\theta}_{\mathrm{NULL}}, \quad \boldsymbol{a} \in \mathscr{A}
\]
where $\mathscr{A}$ denotes the set of admissible contrast vectors.
We proceed with the analysis by the following steps.
\begin{enumerate}
\item Compute the omnibus statistic 
\[
\hat{W}_n^{(\mathscr{A})}
= n\| \Pi_{\hat{T}_n} (\tilde{\boldsymbol{\theta}}_n - \boldsymbol{\theta}_{\mathrm{NULL}}) \|_2^2
= n \sum_{j \in \hat{T}_n} (\tilde{\theta}_{nj} - {\theta}_{j, \mathrm{NULL}})^2,
\]
where ${\theta}_{j, \mathrm{NULL}}$ is the $j$-th component of $\boldsymbol{\theta}_{\mathrm{NULL}}$, and reject $\mathscr{F}$ if $\hat{W}^{(\mathscr{A})}_n/\hat{\sigma}_n^2  > \hat{q}_\alpha$. 
A rejection implies that at least one contrast in $\mathscr{A}$ deviates from zero, while the family-wise error rate is asymptotically controlled.
\item
Assume that the omnibus test in Step~1 is significant.
We may then examine arbitrary post hoc linear contrasts constructed from the selected components.
For any $\boldsymbol{a}\in \mathscr{A}$, the individual hypothesis $H_{0,\boldsymbol{a}}$ is rejected if $n \{ \boldsymbol{a}^\top (\tilde{\boldsymbol{\theta}}_n - \boldsymbol{\theta}_{\mathrm{NULL}} ) \}^2 /\hat\sigma_n^2 > \hat{q}_\alpha$. 
Such post hoc contrasts constructed from selected components can be assessed without further multiplicity correction.
\end{enumerate}
By Lemma~\ref{lem:selection}, any post hoc contrast constructed from the selected components lies in $\mathscr{A}$ asymptotically, so the Scheff\'{e}-type control established in Step~1 extends to all contrasts considered in Step~2.
\end{example}

\begin{example}
As in the test, the Scheff\'{e}-type argument yields a simultaneous confidence band for the regression function
$\boldsymbol{x}^\top \boldsymbol{\theta}_{n}$ over all
$\boldsymbol{x} \in \mathscr{A}$.
Define
\[
I_n(\boldsymbol{x})
=
\left[
\boldsymbol{x}^\top \tilde{\boldsymbol{\theta}}_n
- \left( \frac{\hat{q}_\alpha \hat{\sigma}_n^2}{n} \right)^{1/2}, \ 
\boldsymbol{x}^\top \tilde{\boldsymbol{\theta}}_n
+  \left( \frac{\hat{q}_\alpha \hat{\sigma}_n^2}{n} \right)^{1/2}
\right]
\quad (\boldsymbol{x} \in \R^{p_{n}}).
\]
Then, under Assumptions~\ref{assump:sparsity}--\ref{assump:energy},
\[
\lim_{n\to\infty}\;
\Pr \left(
\boldsymbol{x}^\top \boldsymbol{\theta}_{n}
\in I_n(\boldsymbol{x})
\text{ for all }
\boldsymbol{x} \in \mathscr{A}
\right)
= 1-\alpha.
\]
As $T_\gamma$ is consistently recovered by $\hat{T}_n$, the band permits post hoc confidence statements for the regression function $\boldsymbol{x}^\top \boldsymbol{\theta}_{n}$ for any 
$\boldsymbol{x} \in
\{ \boldsymbol{x} \in \mathbb{R}^{p_{n}} : \operatorname{supp}(\boldsymbol{x}) \subset \hat{T}_n , \|\boldsymbol{x}\|_2 = 1 \}$.
\end{example}

\begin{remark}
Replacing the Euclidean norm with a covariance-weighted norm yields a Hotelling-type statistic, which corresponds to a Wald test for the effective parameter $\boldsymbol{\theta}_{n T_\gamma}$ and converges in distribution to $\chi^2(s_\gamma)$. 
Although such weighting simplifies the null distribution, we use the Euclidean norm to retain a direct connection with the $\ell^2$ weak convergence in Theorem~\ref{thm:ell2_weak}.
\end{remark}

\section{Numerical results}\label{sec:numerical}

This section presents numerical simulations to examine the finite-sample behavior.
Let 
\[(n,p_n,s_{0,n} ) \in \{ (100,100,15), (250,500,20), (500,1000,25) \} .\]
The data-generating process is
\[
Y_{i,n} = \boldsymbol{\theta}_{0,n}^{\top}\boldsymbol{Z}_{i,n} + \epsilon_{i,n} ,
\qquad i \in [n]
\]
where $\{\boldsymbol{Z}_{i,n} \}_{i=1}^n$ are i.i.d.\ $\mathrm{N}_{p_n}(\boldsymbol{0}, \boldsymbol{J}_n )$ random vectors with
\[
\boldsymbol{J}_n = \bigl( 0.3^{|j-k|}\bigr)_{j,k \in [p_n]}
\]
and $\{\epsilon_{i,n} \}_{i=1}^n$ are i.i.d.\ $\mathrm{N}(0, 1)$ random variables independent of $\{\boldsymbol{Z}_{i,n}\}_{i=1}^n$.
We consider the following three cases for $\boldsymbol{\theta}_{0,n}$, throughout which we set $\gamma=1$.
\begin{itemize}
\item[\textbullet]
\textbf{Case 1 (Null)}.
Let
$T_\gamma^{(\mathrm{null})} = [5]$ and $
T_{W,n}^{(\mathrm{null})}= [s_{0,n}] \setminus [5]$.
We set $\boldsymbol{\theta}_{0,n}$ as
\[
\boldsymbol{\theta}^{\mathrm{null}}_{0,nT_\gamma^{(\mathrm{null})}}
= (1.9,-1.7,1.5,-1.3,1.1)^\top,
\]
and
\[
\theta_{0j,n}^{\mathrm{null}}
= n^{-(1+\Delta)/2}(j-5)^{-1},
\quad
j\in T_{W,n}^{(\mathrm{null})}, \quad
\Delta = 2/3,
\]
while $\theta_{0j,n}^{\mathrm{null}}=0$ for $j>s_{0,n}$.
Here, $T_\gamma = T_\gamma^{(\mathrm{null})}$ and $T_{W,n} = T_{W,n}^{(\mathrm{null})}$.
\item[\textbullet]
\textbf{Case 2 (Fixed alternative)}.
We set $\boldsymbol{\theta}_{0,n}$ as
\[
\boldsymbol{\theta}^{\mathrm{alt}}_{0,n}
= \boldsymbol{\theta}^{\mathrm{null}}_{0,n} + \boldsymbol{\delta},
\]
where $\boldsymbol{\delta} = (0.3,-0.3,0,0,0,1.2,0,\ldots,0)^\top$.
Here, $T_\gamma = T_\gamma^{(\mathrm{alt})} = [6]$ and $T_{W,n} = T_{W,n}^{(\mathrm{alt})} = [s_{0,n}]\setminus [6]$.
In this case, the first two coefficients are shifted from their null values, and the sixth component is perturbed by a fixed amount.
\item[\textbullet]
\textbf{Case 3 (Non-negligible weak signals)}.
Let $T_\gamma^{(\mathrm{weak})} = [5]$ and $T_{W,n}^{(\mathrm{weak})} = [s_{0,n}] \setminus [5]$.
We set $\boldsymbol{\theta}_{0,n}$ as
\[
\boldsymbol{\theta}^{\mathrm{weak}}_{0,nT_\gamma^{(\mathrm{weak})}}
= (1.9,-1.7,1.5,-1.3,1.1)^\top,
\]
and
\[
\theta_{0j,n}^{\mathrm{weak}} = 2n^{-1/2}(j-5)^{-1}, \quad
j\in T_{W,n}^{(\mathrm{weak})},
\]
while $\theta_{0j,n}^{\mathrm{weak}}=0$ for $j>s_{0,n}$.
Here, $T_\gamma = T_\gamma^{(\mathrm{weak})}$ and $T_{W,n} = T_{W,n}^{(\mathrm{weak})}$.
In this case, the weak-signal component is non-negligible after scaling by $n^{1/2}$ and converges in $\ell^2$ to a nonzero limit.
This regime corresponds to Theorem~\ref{thm:local_asymp_normal}.
\end{itemize}

For variable selection, we use the lasso estimator with tuning parameter $\lambda_n = n^{-1/2} (\log p_n)^{1/2}$ and set the threshold for constructing $\hat T_n$ in each case to $\tau_n = n^{-1/2} (s_{0,n} \log p_n)^{1/2}$.

For the setting above, we consider testing the linear hypothesis
\[
H_0: \boldsymbol{A}_n \boldsymbol{\theta}_n = \boldsymbol{A}_n \boldsymbol{\theta}^{\mathrm{null}}_{0,n},
\]
where $\boldsymbol{A}_n = (\boldsymbol{e}_1^\top, -\boldsymbol{e}_2^\top)^\top$.
We also consider the Scheff\'{e}-type omnibus test for the family 
\[ \mathscr{F} = \{ H_{0,\boldsymbol{a}} : \boldsymbol{a} \in \mathscr{A}_{n} \}, \quad 
H_{0,\boldsymbol{a}}: \boldsymbol{a}^\top \boldsymbol{\theta}_{n} = \boldsymbol{a}^\top \boldsymbol{\theta}^{\mathrm{null}}_{0,n}, \quad \boldsymbol{a} \in \mathscr{A}_n,
\]
where
$\mathscr{A}_n
= \left\{ \boldsymbol{a} \in \mathbb{R}^{p_{n}} : \operatorname{supp}(\boldsymbol{a}) \subset T_\gamma , \|\boldsymbol{a}\|_2 = 1 \right\}$.
For these tests, we calculate the empirical type-I error rates under Case~1 and the empirical powers under Case~2.
In addition, we construct simultaneous and pointwise confidence intervals for 
$\boldsymbol{v}_j^\top \boldsymbol{\theta}_n$, $j \in [3]$, where
\[
\boldsymbol{v}_1=\boldsymbol{e}_1,\quad 
\boldsymbol{v}_2=(\boldsymbol{e}_1-\boldsymbol{e}_2)/2^{1/2},\quad 
\boldsymbol{v}_3=(\boldsymbol{e}_1-\boldsymbol{e}_2+\boldsymbol{e}_3)/3^{1/2}. 
\]
The simultaneous confidence intervals are obtained by inverting the omnibus test, while the pointwise confidence intervals are constructed as Wald-type intervals.
For all tests, we set the nominal significance level to $0.05$, and all confidence intervals are constructed at the $95\%$ confidence level.
Furthermore, to visualize the asymptotic normality result, we present Q--Q plots of 
$\boldsymbol{u}_j^\top \boldsymbol{\mathcal{R}}_n$, $j \in [3]$, and histograms of $\boldsymbol{u}_3^\top \boldsymbol{\mathcal{R}}_n$, where
\[
\boldsymbol{u}_1 = \boldsymbol{e}_1,\quad 
\boldsymbol{u}_2 = (\boldsymbol{e}_1 + \boldsymbol{e}_2+\boldsymbol{e}_3)/3^{1/2},\quad
\boldsymbol{u}_3 = (\boldsymbol{e}_1 + \boldsymbol{e}_2+\boldsymbol{e}_6+\boldsymbol{e}_7)/2.
\]

\begin{table}[t]
\centering
\caption{Exact selection rates and estimation errors under Case~1.
Values in parentheses are standard deviations.}
\label{tab:null_estimation}
\begin{tabular}{cccccc}
\hline
$n$ & $p_n$ & $s_{0,n}$ & Exact selection rate
& lasso $\ell_2$ error
& Two-step $\ell_2$ error \\
\hline
100 & 100  & 15 & 0.204
& 0.966 (0.195)
& 1.250 (0.594) \\
250 & 500  & 20 & 0.978
& 0.660 (0.100)
& 0.169 (0.161) \\
500 & 1000 & 25 & 1.000
& 0.474 (0.067)
& 0.102 (0.035) \\
\hline
\end{tabular}
\end{table}

\begin{table}[t]
\centering
\caption{Empirical type-I error rates under Case~1.}
\label{tab:null_tests}
\begin{tabular}{ccccc}
\hline
$n$ & $p_n$ & $s_{0,n}$ & Linear & Omnibus \\
\hline
100 & 100  & 15 & 0.112 & 0.230 \\
250 & 500  & 20 & 0.062 & 0.082 \\
500 & 1000 & 25 & 0.054 & 0.056 \\
\hline
\end{tabular}
\end{table}

\begin{table}[t]
\centering
\caption{Empirical coverage probabilities under Case~1.}
\label{tab:null_coverage}
\begin{tabular}{ccccccc}
\hline
$n$ & $p_n$ & $s_{0,n}$
& Simultaneous
& Pointwise $\boldsymbol{v}_1$
& Pointwise $\boldsymbol{v}_2$
& Pointwise $\boldsymbol{v}_3$ \\
\hline
100 & 100  & 15 & 0.770 & 0.936 & 0.900 & 0.826 \\
250 & 500  & 20 & 0.918 & 0.950 & 0.938 & 0.942 \\
500 & 1000 & 25 & 0.944 & 0.956 & 0.956 & 0.932 \\
\hline
\end{tabular}
\end{table}

\begin{table}[t]
\centering
\caption{Average widths of confidence intervals under Case~1.}
\label{tab:null_width}
\begin{tabular}{ccccccc}
\hline
$n$ & $p_n$ & $s_{0,n}$
& Simultaneous
& Pointwise $\boldsymbol{v}_1$
& Pointwise $\boldsymbol{v}_2$
& Pointwise $\boldsymbol{v}_3$ \\
\hline
100 & 100  & 15 & 0.920 & 0.586 & 0.674 & 0.685 \\
250 & 500  & 20 & 0.463 & 0.261 & 0.303 & 0.315 \\
500 & 1000 & 25 & 0.327 & 0.184 & 0.213 & 0.222 \\
\hline
\end{tabular}
\end{table}

\begin{table}[t]
\centering
\caption{Exact selection rates under Cases~2 and~3.}
\label{tab:selection_alt_local}
\begin{tabular}{c c c c c}
\hline
$n$ & $p_n$ & $s_{0,n}$
& Case~2
& Case~3 \\
\hline
100 & 100  & 15 & 0.248 & 0.258 \\
250 & 500  & 20 & 0.984 & 0.982 \\
500 & 1000 & 25 & 1.000 & 1.000 \\
\hline
\end{tabular}
\end{table}

\begin{table}[t]
\centering
\caption{Empirical powers under Case~2.}
\label{tab:fixed_alt}
\begin{tabular}{c c c c c}
\hline
$n$ & $p_n$ & $s_{0,n}$
& Linear
& Omnibus \\
\hline
100 & 100  & 15 & 0.608 & 0.958 \\
250 & 500  & 20 & 0.996 & 1.000 \\
500 & 1000 & 25 & 1.000 & 1.000 \\
\hline
\end{tabular}
\end{table}

\begin{figure}[t]
\centering
\includegraphics[width=0.3\textwidth]{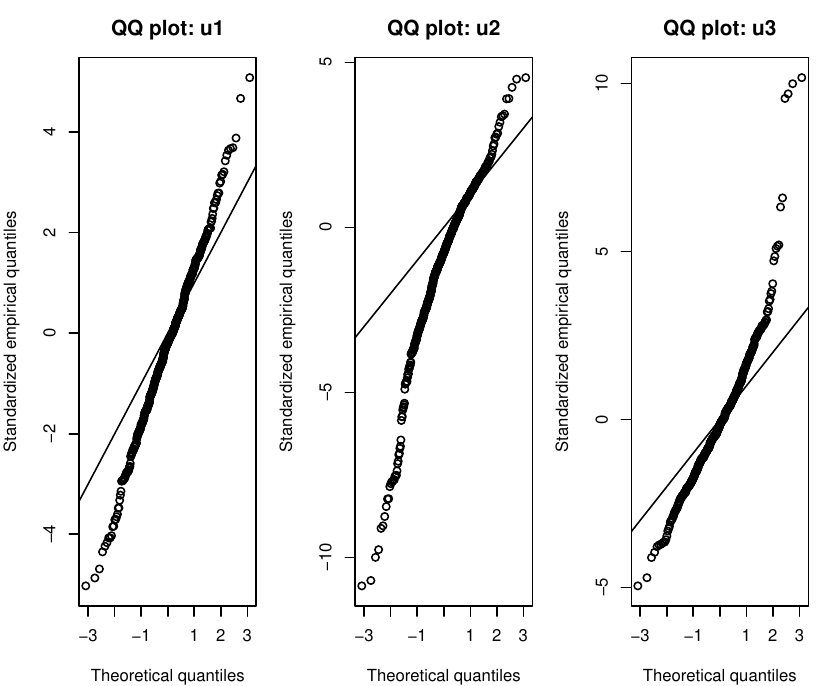}\\
\includegraphics[width=0.3\textwidth]{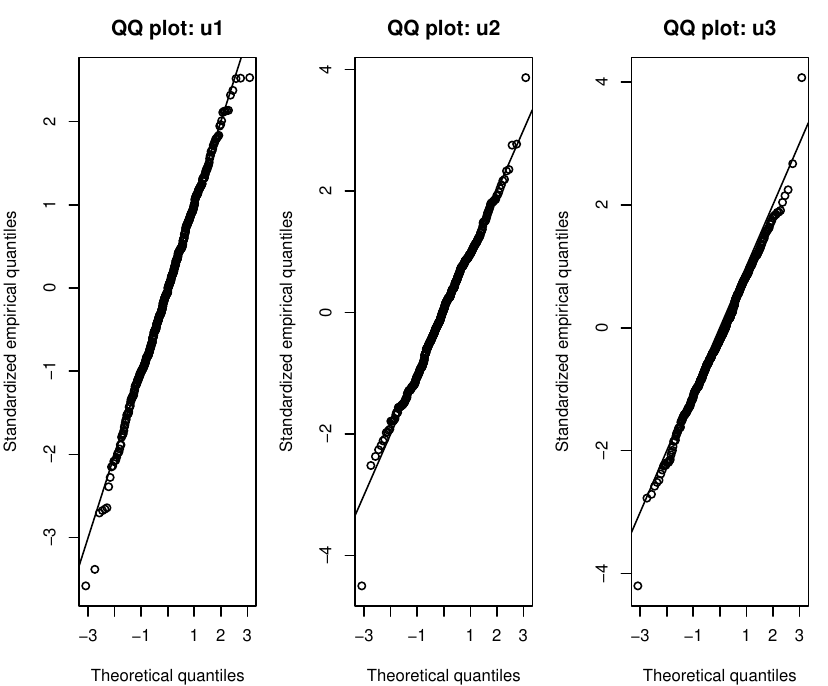}\\
\includegraphics[width=0.3\textwidth]{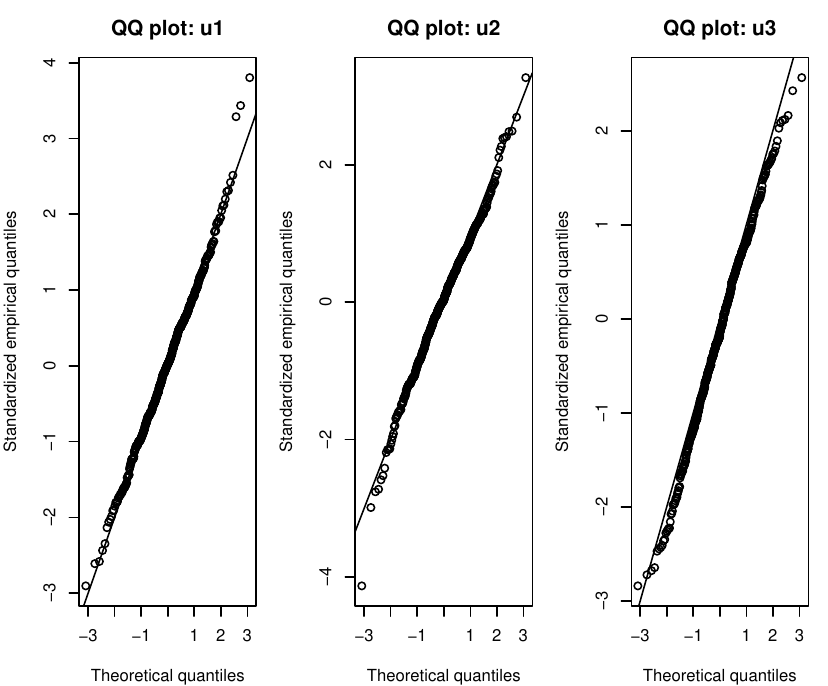}
\caption{Q--Q plots of the standardized values of $\langle \boldsymbol{u}_j, \boldsymbol{\mathcal{R}}_n \rangle$, $j \in [3]$, under Case~1.
From top to bottom, the sample sizes are $n=100$, $250$, and $500$.
From left to right, $j=1$, $2$, and $3$.}
\label{fig:qq_null}
\end{figure}

\begin{figure}[t]
\centering
\includegraphics[width=0.72\textwidth]{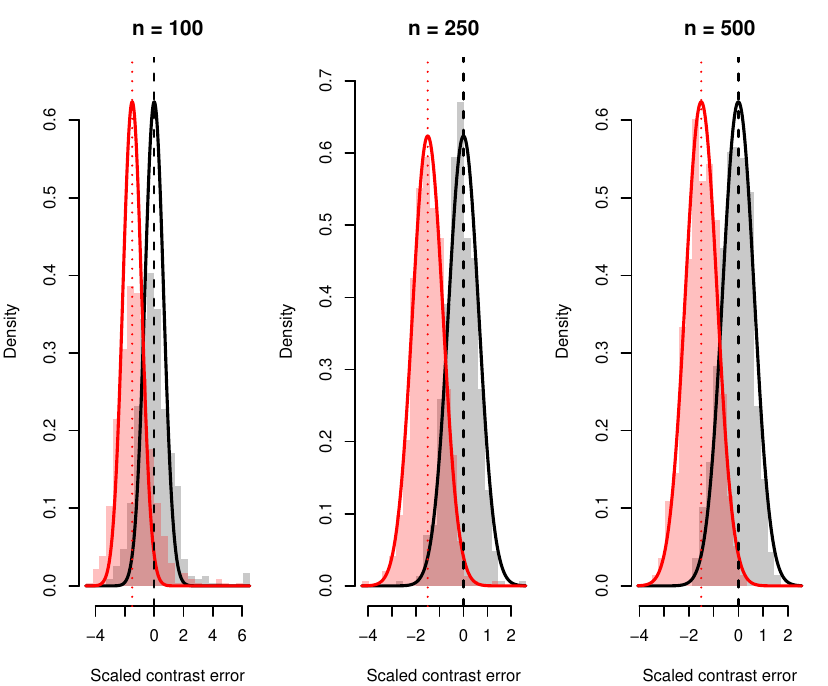}
\caption{Histograms and limiting densities of the variance-normalized values of $\boldsymbol{u}_3^\top \boldsymbol{\mathcal{R}}_n$ under Cases~1 and~3.
Gray bars and black lines correspond to Case~1, and red bars and red lines correspond to Case~3.
From left to right, the sample sizes are $n=100$, $250$, and $500$.}
\label{fig:qq_hist}
\end{figure}

The results based on 500 Monte Carlo replications are presented in Tables~\ref{tab:null_estimation}--\ref{tab:fixed_alt} and Figures~\ref{fig:qq_null} and~\ref{fig:qq_hist}.
The exact selection rate is the proportion of replications for which $\hat{T}_n = T_\gamma$.

Table~\ref{tab:null_estimation} shows the exact selection rate and the $\ell_2$ errors of the lasso and two-step estimators.
The results appear to be consistent with our theoretical results.
The two-step estimator initially has a larger error but outperforms the lasso once selection becomes sufficiently accurate.
Tables~\ref{tab:null_tests} and~\ref{tab:null_coverage} report the empirical sizes and coverage probabilities of the Wald and omnibus procedures.
As $n$ increases, the empirical type-I error rates and coverage probabilities approach their nominal levels, in accordance with our asymptotic results.
As shown in Table~\ref{tab:null_width}, the simultaneous confidence intervals are wider than the pointwise intervals, as expected because the omnibus procedure involves a supremum over a class of contrast vectors.
Tables~\ref{tab:selection_alt_local} and~\ref{tab:fixed_alt} report the exact selection rates under Cases~2 and~3 and the empirical powers under Case~2, respectively.
As expected, these quantities approach one as $n$ increases.

Figure~\ref{fig:qq_null} presents Q--Q plots of the standardized values of
$\boldsymbol{u}_j^\top \boldsymbol{\mathcal{R}}_n$, $j \in [3]$.
When $n=100$, the plots exhibit noticeable deviations from normality, consistent with the low exact selection rate.
As $n$ increases and the selection performance improves, the empirical distributions become closer to the normal distribution.
Moreover, we compare the distributions of the variance-normalized values of $\boldsymbol{u}_3^\top \boldsymbol{\mathcal{R}}_n$ under Cases~1 and~3.
Figure~\ref{fig:qq_hist} shows the histograms together with the corresponding asymptotic normal densities.
The vertical lines indicate the theoretical means, which are zero under Case~1 and determined by the mean-shift term in Theorem~\ref{thm:local_asymp_normal} under Case~3.
The mean shift between the two cases is clearly visible, consistent with Theorem~\ref{thm:local_asymp_normal}.
Since $\boldsymbol{u}_1$ and $\boldsymbol{u}_2$ do not capture the mean-shift direction induced by the weak signals in the present design, their distributions are similar under the two cases, and the corresponding plots are omitted.

\section{Weak convergence with non-negligible weak signals} \label{sec:local}

Let $\{\boldsymbol{h}_n\}_{n \in \mathbb{N}}$ be the deterministic sequence in $\R^\infty$ defined by 
\[
\langle \boldsymbol{e}_j, \boldsymbol{h}_n \rangle
= h_{nj}
= \begin{cases}
n^{1/2}\theta_{0j,n} & (j \in T_{W,n}) \\
0 & (j \in \mathbb{N}\setminus T_{W,n})
\end{cases}.
\]
Then, 
\[ \| \boldsymbol{h}_n \|_2^2 = n \sum_{j \in T_{W,n}} \theta_{0j,n}^2.\]
Under Assumption~\ref{assump:energy}, $\boldsymbol{h}_n \to \boldsymbol{0}$ in $\ell^2$.
We now consider a regime in which the weak signals may have a non-negligible limit.
Recall that
\begin{align*}
&(\boldsymbol{J}_n)_{T_\gamma, T_{W,n}}
= \E[ \boldsymbol{Z}_{1,n T_\gamma} \boldsymbol{Z}_{1,n T_{W,n}}^\top]
,\\
&(\hat{\boldsymbol{J}}_n)_{T_\gamma, T_{W,n}}
= n^{-1} \sum_{i=1}^n \boldsymbol{Z}_{i,n T_\gamma}
\boldsymbol{Z}_{i,n T_{W,n}}^\top.
\end{align*}

\begin{assumption}\label{assump:local_asymp}
\begin{itemize}
\item[(i)]
There exists $\boldsymbol{h} = \lim \boldsymbol{h}_n$ in $\ell^2$.
\item[(ii)]
There exists $\boldsymbol{b}_h = 
\lim \{ (\boldsymbol{J}_n)_{T_\gamma, T_{W,n}} \boldsymbol{h}_{n T_{W,n}} \} 
\in \mathbb{R}^{s_\gamma}$.
\item[(iii)]
$\sup_{n \in \mathbb{N}} \| \boldsymbol{Z}_{1,n T_{0,n}} \|_{\psi_2}  < \infty$.
\end{itemize}
\end{assumption}

\begin{remark}
Assumption~\ref{assump:local_asymp}-(i) allows $\boldsymbol h_n$ to converge in $\ell^2$ to a nonzero limit $\boldsymbol h$. 
Moreover, $\boldsymbol h_{T_\gamma}=\boldsymbol 0$.
\end{remark}

\begin{lemma}\label{lem:local cross cov}
Let Assumptions \ref{assump:sparsity}--\ref{assump:regularity} and \ref{assump:local_asymp} hold.
Then, 
\[
\| (\hat{\boldsymbol{J}}_n)_{T_\gamma, T_{W,n}} \boldsymbol{h}_{n T_{W,n}}
- (\boldsymbol{J}_n)_{T_\gamma, T_{W,n}} \boldsymbol{h}_{n T_{W,n}} \|_2
\to^p 0.
\]
\end{lemma}

\begin{theorem}\label{thm:local_asymp_normal}
Let Assumptions \ref{assump:sparsity}--\ref{assump:regularity} and \ref{assump:local_asymp} hold.
Then, 
\[\boldsymbol{\mathcal{R}}_n \to^d \boldsymbol{\mathcal{R}}_h \quad \text{in} \quad \ell^2, \]
where $\boldsymbol{\mathcal{R}}_h$ is a Gaussian field satisfying
\[ \langle \boldsymbol{u} , \boldsymbol{\mathcal{R}}_h \rangle \sim \mathrm{N}(\boldsymbol{u}_{T_\gamma}^\top ( \boldsymbol{J}_{T_\gamma, T_\gamma} )^{-1} \boldsymbol{b}_h -\langle \boldsymbol{u} , \boldsymbol{h}\rangle,
\sigma^2 \boldsymbol{u}_{T_\gamma}^\top ( \boldsymbol{J}_{T_\gamma, T_\gamma} )^{-1} \boldsymbol{u}_{T_\gamma}) \]
for every $\boldsymbol{u} \in \ell^2$.
\end{theorem}

\section{Proofs}\label{sec:proof}
\subsection{Proofs for Section~\ref{sec:model}}

\begin{proof}[Proof of Lemma~\ref{lem:selection}]
It suffices to show that 
$\hat{T}_n = T_\gamma$
on the event $\{\|\hat{\boldsymbol{\theta}}_n - \boldsymbol{\theta}_{0,n} \|_\infty \leq r_n\}$.

For every $j \in \hat{T}_n^c$, 
we have $|\hat{\theta}_{nj}| \leq \tau_n$ and 
$|\theta_{0j ,n}| - |\hat{\theta}_{nj}|
\leq |\hat{\theta}_{nj} - \theta_{0j {,n}}|
\leq r_n$.
Then, we have $|\theta_{0j ,n}| \leq r_n + \tau_n \leq \gamma$, which implies that $j \in T_\gamma^c$.
Therefore, we obtain $\hat{T}_n^c \subset T_\gamma^c$, or equivalently, 
$T_\gamma \subset \hat{T}_n$.

It holds that 
$T_\gamma^c = T_{W ,n} \cup T_0^c$.
For every $j \in T_0^c$, we have 
$ |\hat{\theta}_{nj}|
= |\hat{\theta}_{nj}-\theta_{0j ,n}|
\leq r_n \leq \tau_n$, which implies that $j \in \hat{T}_n^c$.
Similarly, for every 
$j \in T_{W ,n}$, it holds that 
$|\hat{\theta}_{nj}|
\leq r_n + |\theta_{0j ,n}|
\leq r_n + \eta_{n}
\leq \tau_n$,
which implies that 
$j \in \hat{T}_n^c$.
In conclusion, we have 
$T_\gamma^c \subset \hat{T}_n^c$.
\end{proof}

\begin{proof}[Proof of Lemma~\ref{lem:matrix-approx}]
The conclusion follows from the law of large numbers for triangular arrays.
\end{proof}

\begin{proof}[Proof of Theorem~\ref{thm:ell2_weak}]
It suffices to establish the following (i) and (ii); see, for example, Chapter 1.8 of \cite{vanwellner2023}.
(i) For every $\zeta > 0$,
$\Pr (\|n^{1/2} (\tilde{\boldsymbol{\theta}}_n - \boldsymbol{\theta}_{0 ,n})_{T_\gamma^c}\|_2^2 > \zeta ) \to 0$, which implies that $\{\boldsymbol{\mathcal{R}}_n \}_{n \in \mathbb{N}}$ is asymptotically finite-dimensional.
(ii) For every $\boldsymbol{u} \in \ell^2$,
$\langle \boldsymbol{u}, \boldsymbol{\mathcal{R}}_n \rangle
\to^d \langle \boldsymbol{u}, \boldsymbol{\mathcal{R}} \rangle$.

(i).
It holds that 
\begin{align*}
& \Pr\left(
\|n^{1/2} (\tilde{\boldsymbol{\theta}}_n - \boldsymbol{\theta}_{0 ,n})_{T_\gamma^c}\|_2^2 > \zeta 
\right) \\
& \leq \Pr\left(
n \| (\tilde{\boldsymbol{\theta}}_n - \boldsymbol{\theta}_{0 ,n} )_{T_\gamma^c}\|_2^2 > \zeta , \
\hat{T}_n = T_\gamma \right) 
+ \Pr(\hat{T}_n \neq T_\gamma) .
\end{align*}
The bound
\[
n \| (\tilde{\boldsymbol{\theta}}_n - \boldsymbol{\theta}_{0 ,n } )_{T_\gamma^c}\|_2^2
= n \sum_{j \in T_\gamma^c} |\theta_{0j ,n}|^2 
\leq
n s_{0 ,n} \eta_{n}^2
\]
holds on the event $\{\hat{T}_n = T_\gamma\}$ and hence,
\[
\Pr 
\left(
n \|(\tilde{\boldsymbol{\theta}}_n - \boldsymbol{\theta}_{0 ,n} )_{T_\gamma^c}\|_2^2 > \zeta, \ \hat{T}_n = T_\gamma \right)
\leq 
\Pr\left( n s_{0 ,n} \eta_{n}^2 > \zeta, \ \hat{T}_n = T_\gamma \right) 
\to 0.
\]
Moreover, Lemma~\ref{lem:selection} ensures $\Pr(\hat{T}_n \neq T_\gamma) \to 0$.

(ii).
Fix $\boldsymbol{u} \in \ell^2$.
From
$| \langle \boldsymbol{u}_{T_\gamma^c}, \boldsymbol{\mathcal{R}}_{n T_\gamma^c} \rangle | \leq \|\boldsymbol{u}_{T_\gamma^c}\|_2 \|\boldsymbol{\mathcal{R}}_{n T_\gamma^c}\|_2$
and $\|\boldsymbol{\mathcal{R}}_{n T_\gamma^c}\|_2 = o_p(1)$, the latter being a direct consequence of (i), we obtain
\[
\langle \boldsymbol{u}, \boldsymbol{\mathcal{R}}_n \rangle
= \langle \boldsymbol{u}_{T_\gamma}, \boldsymbol{\mathcal{R}}_{n T_\gamma} \rangle
+\langle \boldsymbol{u}_{T_\gamma^c}, \boldsymbol{\mathcal{R}}_{n T_\gamma^c} \rangle
= \langle \boldsymbol{u}_{T_\gamma}, \boldsymbol{\mathcal{R}}_{n T_\gamma} \rangle
+ o_p(1) .
\]
Hence it suffices to show the convergence in distribution of $\langle \boldsymbol{u}_{T_\gamma}, \boldsymbol{\mathcal{R}}_{n T_\gamma} \rangle$.
To this end, we apply a Lyapunov-type central limit theorem.
For $T \subset [p_{n}]$, let
\[
\psi_{nT}(\boldsymbol{\theta}_{nT})
= n^{-1} \sum_{i=1}^n ( \boldsymbol{J}_{T,T} )^{-1} \boldsymbol{Z}_{i,n T}(Y_{i,n} - \boldsymbol{\theta}_{n T}^\top \boldsymbol{Z}_{i,n T}) .
\]
Then, on the one hand,
\begin{align*} 
&n^{1/2}\psi_{n T_\gamma}(\boldsymbol{\theta}_{0,n T_\gamma}) 1_{\{\hat{T}_n = T_\gamma\}} \\
&= n^{-1/2} \sum_{i=1}^n (\boldsymbol{J}_{T_\gamma, T_\gamma})^{-1} \boldsymbol{Z}_{i,n T_\gamma}(Y_i - \boldsymbol{\theta}_{0,n T_\gamma}^\top \boldsymbol{Z}_{i,n T_\gamma}) 1_{\{\hat{T}_n = T_\gamma\}} \\
&= n^{-1/2} \sum_{i=1}^n (\boldsymbol{J}_{T_\gamma, T_\gamma})^{-1} \boldsymbol{Z}_{i,n T_\gamma}(\tilde{\boldsymbol{\theta}}_{n T_\gamma} - \boldsymbol{\theta}_{0,n T_\gamma} )^\top \boldsymbol{Z}_{i,n T_\gamma} 1_{\{\hat{T}_n = T_\gamma\}} \\
&= ( \boldsymbol{J}_{T_\gamma, T_\gamma} )^{-1} ( \hat{\boldsymbol{J}}_n)_{\hat{T}_n, \hat{T}_n} n^{1/2} (\tilde{\boldsymbol{\theta}}_n - \boldsymbol{\theta}_{0,n})_{\hat{T}_n}1_{\{\hat{T}_n = T_\gamma\}} ,
\end{align*}
where the second equality follows from
\[
n^{-1} \sum_{i=1}^n \boldsymbol{Z}_{i,n \hat{T}_n}(Y_{i,n} - \tilde{\boldsymbol{\theta}}_{n \hat{T}_n}^\top \boldsymbol{Z}_{i ,n \hat{T}_n}) = \boldsymbol{0}.
\]
On the other hand,
\begin{align*}
&n^{1/2}\psi_{n T_\gamma}(\boldsymbol{\theta}_{0,n T_\gamma}) 1_{\{\hat{T}_n = T_\gamma\}} \\ 
&= n^{-1/2} \sum_{i=1}^n (\boldsymbol{J}_{T_\gamma, T_\gamma})^{-1} \boldsymbol{Z}_{i,n T_\gamma}(\boldsymbol{\theta}_{0,n T_\gamma^c}^\top \boldsymbol{Z}_{i,n T_\gamma^c} + \epsilon_{i,n} ) 1_{\{\hat{T}_n = T_\gamma\}} \\
&= 
n^{-1/2} \sum_{i=1}^n (\boldsymbol{J}_{T_\gamma, T_\gamma})^{-1} \boldsymbol{Z}_{i,n T_\gamma} \boldsymbol{Z}_{i,n T_\gamma^c}^\top \boldsymbol{\theta}_{0,n T_\gamma^c} 1_{\{\hat{T}_n = T_\gamma\}} 
\\ &\quad
+n^{-1/2} \sum_{i=1}^n (\boldsymbol{J}_{T_\gamma, T_\gamma})^{-1} \boldsymbol{Z}_{i,n T_\gamma} \epsilon_{i,n} 1_{\{\hat{T}_n = T_\gamma\}} .
\end{align*}
For each $j \in \{1,\ldots, s_\gamma \}$, the $j$-th component of first term on the right-hand side is evaluated as
\begin{align*}
& | n^{-1/2} \sum_{i=1}^n ( (\boldsymbol{J}_{T_\gamma, T_\gamma})^{-1} \boldsymbol{Z}_{i,n T_\gamma} )_j \boldsymbol{Z}_{i,n T_\gamma^c}^\top \boldsymbol{\theta}_{0,n T_\gamma^c} | \\
& \leq 
n^{-1/2} \sum_{i=1}^n |( (\boldsymbol{J}_{T_\gamma, T_\gamma})^{-1} \boldsymbol{Z}_{i,n T_\gamma} )_j | |  \boldsymbol{Z}_{i,n T_\gamma^c}^\top \boldsymbol{\theta}_{0,n T_\gamma^c} | \\
& \leq 
\left( n^{-1} \sum_{i=1}^n |( (\boldsymbol{J}_{T_\gamma, T_\gamma})^{-1} \boldsymbol{Z}_{i,n T_\gamma} )_j |^2 \right)^{1/2}
\left( \sum_{i=1}^n |  \boldsymbol{Z}_{i,n T_\gamma^c}^\top \boldsymbol{\theta}_{0,n T_\gamma^c} |^2 \right)^{1/2}.
\end{align*}
The first factor is $O_p(1)$, whereas the second satisfies
\begin{align*}
&n \E [ |  \boldsymbol{Z}_{1,n T_\gamma^c}^\top \boldsymbol{\theta}_{0,n T_\gamma^c} |^2 ]
= n \E \bigl[ \bigl|  \sum_{j \in T_{W,n}} Z_{1j,n} \theta_{0 j, n} \bigr|^2 \bigr] \\
&\leq n \|  \boldsymbol{Z}_{1 T_{W,n}} \|_{\psi_2}^2 \| \boldsymbol{\theta}_{0,n T_{W,n}}  \|_2^2 
\leq n s_{0,n}^{1 + 2\kappa} \eta_n^2 (s_{0,n}^{-\kappa} \|  \boldsymbol{Z}_{1,n T_{0,n}} \|_{\psi_2})^2
\to 0.
\end{align*}
Consequently, 
\begin{align*}
&n^{-1/2} \sum_{i=1}^n (\boldsymbol{J}_{T_\gamma, T_\gamma})^{-1} \boldsymbol{Z}_{i,n T_\gamma} \boldsymbol{Z}_{i,n T_\gamma^c}^\top \boldsymbol{\theta}_{0,n T_\gamma^c} 1_{\{\hat{T}_n = T_\gamma\}} \\
&= O_p(n^{1/2} s_{0,n}^{1/2+\kappa} \eta_n)
= o_p(1).
\end{align*}
From
\begin{align*}
& \| ( \hat{\boldsymbol{J}}_n)_{ T_\gamma, T_\gamma} - \boldsymbol{J}_{T_\gamma, T_\gamma} \|_{\mathrm{op}} \\
&\leq \| ( \hat{\boldsymbol{J}}_n)_{T_\gamma, T_\gamma} - (\boldsymbol{J}_n)_{T_\gamma, T_\gamma} \|_{\mathrm{op}} + \ | (\boldsymbol{J}_n)_{T_\gamma, T_\gamma} - \boldsymbol{J}_{T_\gamma, T_\gamma} \|_{\mathrm{op}} = o_p(1)
\end{align*} 
and 
\[ 1_{\{\hat{T}_n = T_\gamma\}}=1+o_p(1),\]
it follows that
\[
\boldsymbol{u}_{T_\gamma}^\top {n}^{1/2} (\tilde{\boldsymbol{\theta}}_n - \boldsymbol{\theta}_{0,n} )_{T_\gamma} 
=
n^{-1/2} \sum_{i=1}^n \boldsymbol{u}_{T_\gamma}^\top (\boldsymbol{J}_{T_\gamma, T_\gamma})^{-1} \boldsymbol{Z}_{i,n T_\gamma} \epsilon_{i,n} 
+ o_p(1).
\]
Moreover,
\begin{align*}
& n^{-1} \sum_{i=1}^n \boldsymbol{u}_{T_\gamma}^\top  ( \boldsymbol{J}_{T_\gamma, T_\gamma} )^{-1} \boldsymbol{Z}_{i,n  T_\gamma} \boldsymbol{Z}_{i,n T_\gamma}^\top ( \boldsymbol{J}_{T_\gamma, T_\gamma} )^{-1} \boldsymbol{u}_{T_\gamma} \epsilon_{i,n }^2 \\
&\to^p \sigma^2 \boldsymbol{u}_{T_\gamma}^\top ( \boldsymbol{J}_{T_\gamma, T_\gamma} )^{-1} \boldsymbol{u}_{T_\gamma} .
\end{align*}
As $s_\gamma$ is independent of $n$ and $p_n$, 
\[ \E[|\boldsymbol{u}_{T_\gamma}^\top ( \boldsymbol{J}_{T_\gamma, T_\gamma} )^{-1} \boldsymbol{Z}_{1 ,n T_\gamma}|^{2+\delta}] = O(1), \quad
\E[|\epsilon_{1 ,n }|^{2+\delta}] = O(1) \]
for any $\delta>0$.
Thus,
\begin{align*}
&\sum_{i=1}^n \E[|
n^{-1/2} \boldsymbol{u}_{T_\gamma}^\top ( \boldsymbol{J}_{T_\gamma, T_\gamma} )^{-1} \boldsymbol{Z}_{i,n T_\gamma}\epsilon_{i,n} |^{2 + \delta}
] \\
&\leq 
n^{-\delta/2}\E[|\boldsymbol{u}_{T_\gamma}^\top ( \boldsymbol{J}_{T_\gamma, T_\gamma} )^{-1} \boldsymbol{Z}_{1,n T_\gamma}|^{2+\delta}]\E[ |\epsilon_{1,n } |^{2+\delta}]
= o(1).
\end{align*}
Therefore,
\begin{align*}
\langle \boldsymbol{u}_{T_\gamma}, \boldsymbol{\mathcal{R}}_{n T_\gamma} \rangle 
&= {n}^{1/2} \boldsymbol{u}_{T_\gamma}^\top \psi_{n T_\gamma}(\boldsymbol{\theta}_{0 ,n T_\gamma}) + o_p(1) \\
&\to^d \mathrm{N} (0, \sigma^2 \boldsymbol{u}_{T_\gamma}^{\top} ( \boldsymbol{J}_{T_\gamma, T_\gamma} )^{-1} \boldsymbol{u}_{T_\gamma}) .
\end{align*}
\end{proof}

\subsection{Proofs for Section~\ref{sec:appli1}}

\begin{proof}[Proof of Corollary~\ref{cor:test asymptotic distribution}]
Suppose that $H_0$ is true.
Define the sequence $\{ \mathcal{T}_{A_n} \}_{n \in \N}$ of linear operators $\mathcal{T}_{A_n}: \ell^2 \to \ell^2$ and a linear operator $\mathcal{T}_{A}: \ell^2 \to \ell^2$ by
\[
\langle \boldsymbol{e}_j, \mathcal{T}_{A_n} (\boldsymbol{x}) \rangle 
= \begin{cases}
\langle \tilde{\boldsymbol{a}}_{j,n}, \boldsymbol{x} \rangle & (j \in [q])\\
0 & (j \in \mathbb{N}\setminus[q])
\end{cases}
\]
and
\[
\langle \boldsymbol{e}_j, \mathcal{T}_A(\boldsymbol{x}) \rangle 
= \begin{cases}
\langle \boldsymbol{a}_j, \boldsymbol{x} \rangle & (j \in [q])\\
0 & (j \in \mathbb{N}\setminus[q])
\end{cases}
\]
for $\boldsymbol{x} \in \ell^2$, respectively.
Consider the sequence of the $\ell^2$-valued random variables $\mathcal{T}_{A_n} (\boldsymbol{\mathcal{R}}_n)$.
As $n\to\infty$,
\[ \mathcal{T}_{A_n}(\boldsymbol{x}_n) 
= \sum_{j = 1}^q \langle \tilde{\boldsymbol{a}}_{j,n}, \boldsymbol{x}_n \rangle \boldsymbol{e}_j 
\to \sum_{j = 1}^q \langle \boldsymbol{a}_{j}, \boldsymbol{x} \rangle \boldsymbol{e}_j 
= \mathcal{T}_{A}(\boldsymbol{x}) \]
for any sequence $\{\boldsymbol{x}_n\}_{n \in \N}$ such that $\boldsymbol{x}_n \in \ell^2$ for all $n \in \N$ and $\boldsymbol{x}_n \to \boldsymbol{x} \in \ell^2$.
It follows from Theorem~\ref{thm:ell2_weak} and the extended continuous mapping theorem \cite[Theorem~1.11.1]{vanwellner2023} that 
\[ \mathcal{T}_{A_n} (\boldsymbol{\mathcal{R}}_n) \to^d  \mathcal{T}_A (\boldsymbol{\mathcal{R}}) \quad \text{in} \quad  \ell^2. \]
The limit $\mathcal{T}_A (\boldsymbol{\mathcal{R}})$ is an $\ell^2$-valued random variable satisfying 
\[ \langle \boldsymbol{u}, \mathcal{T}_A (\boldsymbol{\mathcal{R}}) \rangle \sim \mathrm{N}(0, \sigma^2 \boldsymbol{u}_{[q]}^\top \boldsymbol{\Sigma}_A \boldsymbol{u}_{[q]}) \]
for any $\boldsymbol{u} \in \ell^2$.
Consequently,  
\[W_n = \| \mathcal{T}_{A_n} (\boldsymbol{\mathcal{R}}_n) \|_2^2 \to^d \| \mathcal{T}_A (\boldsymbol{\mathcal{R}})\|_2^2.\]
Since
$\| \mathcal{T}_A (\boldsymbol{\mathcal{R}}) \|_2^2 =^d \|\boldsymbol{R}_A\|_2^2$ with $\boldsymbol{R}_A \sim \mathrm{N}_q(\boldsymbol{0}, \sigma^2 \boldsymbol{\Sigma}_A)$, a spectral decomposition of $\boldsymbol{\Sigma}_A$ yields
\[ \| \mathcal{T}_A (\boldsymbol{\mathcal{R}}) \|_2^2 =^d
\sigma^2 \sum_{j=1}^r \Lambda_j \chi_{1,j}^2. \]
\end{proof}

\begin{proof}[Proof of Lemma~\ref{lem:sigma_estimation}]
Decompose $\hat{\sigma}_n^2$ as 
\[ \hat{\sigma}_n^2 = S_{n1} -2 S_{n2} + S_{n3} + S_{n4} + 2 S_{n5} - 2 S_{n6}, \] 
where
\begin{align*}
&S_{n1} := n^{-1} \sum_{i=1}^n \epsilon_{i ,n}^2,\\
&S_{n2} := n^{-1} \sum_{i=1}^n (\tilde{\boldsymbol{\theta}}_n - \boldsymbol{\theta}_{0, n})_{\hat{T}_n}^\top \boldsymbol{Z}_{i,n \hat{T}_n} \epsilon_{i,n},\\
&S_{n3} := n^{-1} \sum_{i=1}^n \left\{
(\tilde{\boldsymbol{\theta}}_n - \boldsymbol{\theta}_{0,n})_{\hat{T}_n}^\top \boldsymbol{Z}_{i,n \hat{T}_n}
\right\} ^2 , \\
&S_{n4} := n^{-1} \sum_{i=1}^n ( {\boldsymbol{\theta}}_{0,n \hat{T}_n^c}^\top \boldsymbol{Z}_{i,n \hat{T}_n^c} )^2, \\
&S_{n5} := n^{-1} \sum_{i=1}^n (\boldsymbol{\theta}_{0,n \hat{T}_n^c}^\top \boldsymbol{Z}_{i,n \hat{T}_n^c} ) \epsilon_{i,n} , \\
&S_{n6} := n^{-1} \sum_{i=1}^n \{ (\tilde{ \boldsymbol{\theta}}_{n \hat{T}_n} - \boldsymbol{\theta}_{0,n} )_{\hat{T}_n}^\top \boldsymbol{Z}_{i,n \hat{T}_n} \}  {\boldsymbol{\theta}}_{0,n \hat{T}_n^c}^\top \boldsymbol{Z}_{i,n \hat{T}_n^c}.
\end{align*}
By the law of large numbers for triangular arrays, $S_{n1} \to^p \sigma^2$.
To prove $S_{n2} = o_p(1)$, it suffices to show that 
\[
\Pr\left(
n^{-1} \left|
\sum_{i=1}^n (\tilde{\boldsymbol{\theta}}_n - \boldsymbol{\theta}_{0,n})_{\hat{T}_n}^\top \boldsymbol{Z}_{i,n \hat{T}_n} \epsilon_{i,n}
\right| > \zeta
\right)\to0,
\]
for any $\zeta>0$.
As in the proof of Theorem~\ref{thm:ell2_weak}, the left-hand side admits the bound
\begin{align*}
&\Pr \left( n^{-1} \left| \sum_{i=1}^n (\tilde{\boldsymbol{\theta}}_n - \boldsymbol{\theta}_{0, n})_{T_\gamma}^\top \boldsymbol{Z}_{i ,n T_\gamma} \epsilon_{i,n} \right| > \zeta, \ \hat{T}_n = T_\gamma\right) 
+ \Pr(\hat{T}_n \neq T_\gamma) \\
&\leq
\Pr\left(
n^{-1}\left|
\sum_{i=1}^n (\tilde{\boldsymbol{\theta}}_n - \boldsymbol{\theta}_{0,n})_{T_\gamma}^\top \boldsymbol{Z}_{i,n T_\gamma} \epsilon_{i,n}
\right| > \zeta
\right) + \Pr(\hat{T}_n \neq T_\gamma).    
\end{align*}
Since $n^{1/2}\|\tilde{\boldsymbol{\theta}}_{n} - \boldsymbol{\theta}_{0,n} \|_2$ is $O_p(1)$ and $s_\gamma$ is fixed,
\begin{align*}
& n^{-1} \left|\sum_{i=1}^n (\tilde{\boldsymbol{\theta}}_n - \boldsymbol{\theta}_{0,n})_{T_\gamma}^\top \boldsymbol{Z}_{i,n T_\gamma} \epsilon_{i,n}
\right| \\
& \leq \|(\tilde{\boldsymbol{\theta}}_n - \boldsymbol{\theta}_{0,n})_{T_\gamma}\|_2
\left\| n^{-1} 
\sum_{i=1}^n  \boldsymbol{Z}_{i,n T_\gamma} \epsilon_{i,n}
\right\|_2
= o_p(1) O_p(1) = o_p(1).
\end{align*}
Moreover, Lemma~\ref{lem:selection} implies $\Pr(\hat{T}_n \neq T_\gamma) =o(1)$.
To prove $S_{n3} = o_p(1)$, it suffices to show that 
\[
n^{-1} \sum_{i=1}^n \{
(\tilde{\boldsymbol{\theta}}_n - \boldsymbol{\theta}_{0,n})_{T_\gamma}^\top \boldsymbol{Z}_{i,n T_\gamma}
\}^2 = o_p(1),
\]
which follows from
\begin{align*}
& n^{-1}  \sum_{i=1}^n \{
(\tilde{\boldsymbol{\theta}}_n - \boldsymbol{\theta}_{0,n})_{T_\gamma}^\top \boldsymbol{Z}_{i,n T_\gamma}
 \}^2 \\
& \leq \|(\tilde{\boldsymbol{\theta}}_n - \boldsymbol{\theta}_{0,n})_{T_\gamma}\|_2^2 
n^{-1} \sum_{i=1}^n  \|\boldsymbol{Z}_{i,n T_\gamma} \|_2^2
= o_p(1) O_p(1) = o_p(1).
\end{align*}
To prove $S_{n4} = o_p(1)$, by Lemma~\ref{lem:selection}, it suffices to show that 
\[ \E[S_{n4} 1_{\{ \hat{T}_n = T_\gamma \}}] \to 0. \]
Indeed,
\begin{align*}
\E[S_{n4} 1_{\{ \hat{T}_n = T_\gamma \}}] 
&= n^{-1} \sum_{i=1}^n \E[(\boldsymbol{\theta}_{0,n T_{W,n}}^\top \boldsymbol{Z}_{i,n T_{W,n}})^2 1_{\{ \hat{T}_n = T_\gamma \}}]\\
&\le n^{-1} \sum_{i=1}^n \|\boldsymbol{\theta}_{0,n T_{W,n}}\|_2^2
\|\boldsymbol{Z}_{i,n T_{W,n}} \|_{\psi_2}^2 \\
&\leq 
s_{0,n}^{1+2\kappa}\eta_n^2 (s_{0,n}^{-\kappa}\|\boldsymbol{Z}_{1,n T_{0,n}} \|_{\psi_2})^2 
= o(1).
\end{align*}
By the Cauchy--Schwarz inequality, $S_{n5}^2\leq S_{n1}S_{n4}$ and $S_{n6}^2\leq S_{n3}S_{n4}$. 
Thus, $S_{n5}=o_p(1)$ and $S_{n6}=o_p(1)$.
\end{proof}

\begin{proof}[Proof of Lemma~\ref{lem:lambda_estimation}]
Let $\tilde{\Lambda}_{n,j}$ be the $j$-th largest eigenvalue of
\[ 
\tilde{\boldsymbol{\Sigma}}_{n A} = 
(\boldsymbol{A}_n)_{[q], T_\gamma} \{( \hat{\boldsymbol{J}}_n )_{T_\gamma, T_\gamma} \}^{-1} (\boldsymbol{A}_n)_{[q], T_\gamma}^\top \]
for $j \in [r]$.
On the event $\{\hat{T}_n = T_\gamma\}$, we have $\hat{\Lambda}_{n, j} = \tilde{\Lambda}_{n, j}$, which implies
\[ \max_{j \in [r]}|\hat{\Lambda}_{n, j} - \Lambda_j|1_{\{\hat{T}_n = T_\gamma\}} 
=  \max_{j \in [r]}|\tilde{\Lambda}_{n, j} - \Lambda_j|1_{\{\hat{T}_n = T_\gamma\}} 
\leq \max_{j \in [r]}|\tilde{\Lambda}_{n, j} - \Lambda_j| . \]
By Weyl's inequality and Lemma~\ref{lem:matrix-approx}, 
\begin{align*}
& \max_{j \in [r]}|\tilde{\Lambda}_{n, j} - \Lambda_{j}| \\
& \leq \|\tilde{\boldsymbol{\Sigma}}_{n A} - \boldsymbol{\Sigma}_{A } \|_{\mathrm{op}}
\leq \|\tilde{\boldsymbol{\Sigma}}_{n A} - \boldsymbol{\Sigma}_{n A} \|_{\mathrm{op}}  + \|\boldsymbol{\Sigma}_{n A} - \boldsymbol{\Sigma}_{A} \|_{\mathrm{op}}  = o_p(1), 
\end{align*}
where $\boldsymbol{\Sigma}_{n A} = (\boldsymbol{A}_n)_{[q], T_\gamma} \{(\boldsymbol{J}_n)_{T_\gamma, T_\gamma} \}^{-1} (\boldsymbol{A}_n)_{[q], T_\gamma}^\top $.
\end{proof}

\begin{proof}[Proof of Theorem~\ref{thm:test_rule1}]
Suppose that $H_0$ is true.
It holds that
\begin{align*}
&\sum_{j=1}^{r} \Lambda_j \chi_{1,j}^2 - \sum_{j=1}^{\hat{r}_n} \hat{\Lambda}_{n,j} \chi_{1,j}^2 \\
&=
\left( \sum_{j=1}^{r} \Lambda_j \chi_{1,j}^2 - \sum_{j=1}^{\hat{r}_n} \hat{\Lambda}_{n,j} \chi_{1,j}^2 \right) 1_{\{\hat{T}_n = T_\gamma\}} \\
&\qquad + \left( \sum_{j=1}^{r} \Lambda_j \chi_{1,j}^2 - \sum_{j=1}^{\hat{r}_n} \hat{\Lambda}_{n,j} \chi_{1,j}^2 \right) 1_{\{\hat{T}_n \neq T_\gamma\}} \\
&=
\sum_{j=1}^{r} (\Lambda_j - \hat{\Lambda}_{n,j}) 1_{\{\hat{T}_n = T_\gamma\}} \chi_{1,j}^2 
+ \left( \sum_{j=1}^{r} \Lambda_j \chi_{1,j}^2 - \sum_{j=1}^{\hat{r}_n} \hat{\Lambda}_{n,j} \chi_{1,j}^2 \right) 1_{\{\hat{T}_n \neq T_\gamma\}}.
\end{align*}
By Lemma~\ref{lem:lambda_estimation},
\[
\left| \sum_{j=1}^{r} (\Lambda_j - \hat{\Lambda}_{n,j}) 1_{\{\hat{T}_n = T_\gamma\}} \chi_{1,j}^2  \right|
\leq \max_{j \in [r]} |\Lambda_j - \hat{\Lambda}_{n,j}|  1_{\{\hat{T}_n = T_\gamma\}} \sum_{j=1}^{r}  \chi_{1,j}^2 \to^p 0
\]
and, by Lemma~\ref{lem:selection},
\[
\left| \sum_{j=1}^{r} \Lambda_j \chi_{1,j}^2 - \sum_{j=1}^{\hat{r}_n} \hat{\Lambda}_{n,j} \chi_{1,j}^2 \right| 1_{\{\hat{T}_n \neq T_\gamma\}}
\to^p 0.
\]
Since $\|\boldsymbol{R}_A\|_2^2$ has a continuous distribution function, $\hat{c}_\alpha \to^p c_\alpha$ by Lemma~\ref{lem:lambda_estimation}.
Hence, by Lemma~\ref{lem:sigma_estimation} and Slutsky's lemma, 
\[ (W_n/\hat{\sigma}_n^2 , \hat{c}_\alpha)\to^d (\|\boldsymbol{R}_A\|_2^2/\sigma^2, c_\alpha).\]
As $\Pr(\|\boldsymbol{R}_A\|_2^2/\sigma^2 = c_\alpha)=0$, 
\[ \varphi_n \to^d 1_{\{\|\boldsymbol{R}_A\|_2^2/\sigma^2 > c_\alpha\}} \]
by the continuous mapping theorem.
Finally, by the Portmanteau theorem,
\[ \Pr(\varphi_n = 1) \to \Pr(\|\boldsymbol{R}_A\|_2^2 / \sigma^2 > c_\alpha) = \alpha. \]
\end{proof}

\subsection{Proofs for Section~\ref{sec:appli2}}

\begin{proof}[Proof of Lemma~\ref{lem:proj}]
Fix $\boldsymbol{a}\in\mathscr{A}$.
The orthogonal projection of $\tilde{\boldsymbol\theta}_n$ onto $H(\boldsymbol a)$ is
\[
\Pi_{H(\boldsymbol{a})} \tilde{\boldsymbol\theta}_n
= \tilde{\boldsymbol\theta}_n-
(\boldsymbol{a}^\top \tilde{\boldsymbol\theta}_n-c(\boldsymbol{a})) \boldsymbol{a}.
\]
Hence,
\[ \| \tilde{\boldsymbol{\theta}}_n -  \Pi_{H(\boldsymbol{a})} \tilde{\boldsymbol{\theta}}_n \|_2^2 
= (\boldsymbol{a}^\top \tilde{\boldsymbol{\theta}}_n - c(\boldsymbol{a}))^2 
= [ \boldsymbol{a}^\top (\tilde{\boldsymbol{\theta}}_n - \boldsymbol{\theta}_{0,n} )]^2. \]
Since $\operatorname{supp}(\boldsymbol{a}) \subset T_\gamma$ and $\|\boldsymbol a\|_2=1$ for any $\boldsymbol{a} \in \mathscr{A}$, taking the supremum over $\boldsymbol{a} \in \mathscr{A}$ yields 
\[ 
\sup_{\boldsymbol{a}\in \mathscr{A}} \{ \boldsymbol{a}^\top (\tilde{\boldsymbol{\theta}}_n - \boldsymbol{\theta}_{0,n} )\}^2 
= \sup_{\boldsymbol{a}\in \mathscr{A}} \{ \boldsymbol{a}^\top \Pi_{T_\gamma}(\tilde{\boldsymbol\theta}_n-\boldsymbol\theta_{0,n})\}^2 
= \| \Pi_{T_\gamma} ( \tilde{\boldsymbol{\theta}}_n - \boldsymbol{\theta}_{0,n}) \|_2^2. \]
\end{proof}

\begin{proof}[Proof of Corollary~\ref{cor:test2_asydist}]
Suppose that the family $\mathscr F$ is true.
Define a linear operator $\mathcal{T}_\gamma: \ell^2 \to \ell^2$
by 
\[
\langle \boldsymbol{e}_j, \mathcal{T}_\gamma(\boldsymbol{x}) \rangle
= \begin{cases}
\langle \boldsymbol{e}_j, \boldsymbol{x} \rangle & (j \in T_\gamma) \\
0 & (j \in \mathbb{N}\setminus T_\gamma)
\end{cases}
\]
for $\boldsymbol{x} \in \ell^2$.
From Theorem~\ref{thm:ell2_weak} and the continuous mapping theorem,
\[ \mathcal{T}_\gamma(\boldsymbol{\mathcal{R}}_n) \to^d \mathcal{T}_\gamma(\boldsymbol{\mathcal{R}}) \quad \text{in} \ \ell^2. \]
The remaining steps are analogous to those in the proof of Corollary~\ref{cor:test asymptotic distribution}.
\end{proof}

\begin{proof}[Proof of Theorem~\ref{thm:test_rule2}]
Suppose that the family $\mathscr{F}$ is true.
As in Lemma~\ref{lem:lambda_estimation},
\[ \max_{j \in [s_\gamma]}|\hat{\lambda}_{n, j} - \lambda_j| 1_{\{\hat{T}_n = T_\gamma\}} \to^p 0.\]
Due to Lemma~\ref{lem:selection}, this implies
\[
\sum_{j=1}^{s_\gamma} \lambda_j \chi_{1,j}^2 - \sum_{j=1}^{\hat{s}_\gamma} \hat{\lambda}_{n, j} \chi_{1,j}^2 
\to^p 0.
\]
The rest of the proof is straightforward by applying Slutsky's lemma and the continuous mapping theorem, as in the proof of Theorem~\ref{thm:test_rule1}.
\end{proof}

\subsection{Proofs for Section~\ref{sec:local}}
\begin{proof}[Proof of Lemma \ref{lem:local cross cov}]
Fix $k \in T_\gamma$.
Define
\[
\xi_{i,k,n} = Z_{ik,n}(\boldsymbol{Z}_{i,n T_{W,n}}^\top \boldsymbol{h}_{n T_{W,n}})
-\E[Z_{ik,n}(\boldsymbol{Z}_{i,n T_{W,n}}^\top \boldsymbol{h}_{n T_{W,n}})],\quad i \in [n]
\]
and
\[
X_{n,k} = n^{-1}\sum_{i=1}^n \xi_{i,k,n}.
\]
Then,
\[
X_{n,k}
= (\hat{\boldsymbol{J}}_{n T_\gamma, T_{W,n}} \boldsymbol{h}_{n T_{W,n}}
- \boldsymbol{J}_{n T_\gamma, T_{W,n}} \boldsymbol{h}_{n T_{W,n}})_k.
\]
As $T_\gamma$ is a finite set, it suffices to show $X_{n,k} \to^p 0$.
Since $\{\xi_{i,k,n}\}_{i \in [n]}$ is independent and identically distributed with mean zero, we have 
\[
\E[X_{n,k}^2]
= n^{-1} \E[\xi_{1,k,n}^2] \to 0,
\]
because
\begin{align*}
\E[\xi_{1,k,n}^2]
&\leq
\E[ Z_{1k,n}^2(\boldsymbol{Z}_{1,n T_{W,n}}^\top \boldsymbol{h}_{n,T_{W,n}})^2] \\
&\leq \left\{ \E[ Z_{1k,n}^4] \E[(\boldsymbol{Z}_{1,n T_{W,n}}^\top \boldsymbol{h}_{n, T_{W,n}})^4] \right\}^{1/2} \\
&\leq 2 \left(\E[ Z_{1k,n}^4] \|\boldsymbol{Z}_{1,n T_{W,n}}\|_{\psi_2}^4 \right)^{1/2} \|\boldsymbol{h}_{n, T_{W,n}}\|_2^2
=O(1).
\end{align*}
The conclusion follows from Chebyshev's inequality.
\end{proof}

\begin{proof}[Proof of Theorem~\ref{thm:local_asymp_normal}]
The proof follows that of Theorem~\ref{thm:ell2_weak}, except that the weak-signal component is no longer negligible and induces a mean shift in the limit.
\begin{itemize}
\item[(i).]
To show that the sequence is asymptotically finite-dimensional, it suffices to verify
\begin{equation}\label{eq:afd}
\lim_{m \to \infty}\limsup_{n \to \infty}
\Pr\left(
\sum_{j > m} |n^{1/2}(\tilde{\boldsymbol{\theta}}_n - \boldsymbol{\theta}_{0,n})_j|^2 > \zeta
\right) = 0
\end{equation}
for every $\zeta>0$.
For $m$ such that $m > \max T_\gamma$, 
\begin{align*}
& \Pr\left(
\sum_{j > m} |n^{1/2}(\tilde{\boldsymbol{\theta}}_n - \boldsymbol{\theta}_{0,n})_j|^2 > \zeta
\right) \\
&\leq
\Pr\left(
\sum_{j > m} |n^{1/2}(\tilde{\boldsymbol{\theta}}_n - \boldsymbol{\theta}_{0,n})_j|^2 > \zeta , \
\hat{T}_n = T_\gamma \right) 
+ \Pr(\hat{T}_n \neq T_\gamma) \\
& \leq
\Pr\left(
\sum_{j > m} |h_{nj}|^2 1_{\{\hat{T}_n = T_\gamma\}} > \zeta
\right) + o(1) \\
& \leq
\frac{ \Pr(
 \hat{T}_n = T_\gamma)}{\zeta} \sum_{j > m} |h_{nj}|^2   + o(1).
\end{align*}
Since $\boldsymbol h_n\to\boldsymbol h = (h_1,h_2,\ldots)^\top$ in $\ell^2$,
\begin{align*}
\lim_{m \to \infty}\limsup_{n \to \infty} \sum_{j >m} |h_{nj}|^2
&\leq 2\lim_{m \to \infty}\limsup_{n \to \infty} \left\{ \sum_{j >m} |h_{nj}-h_j|^2 + 2 \sum_{j >m} |h_j|^2 \right\}\\
&\leq 2 \lim_{m \to \infty} \sum_{j >m} |h_j|^2
= 0,   
\end{align*}
from which \eqref{eq:afd} follows.

\item[(ii).]
Fix $\boldsymbol{u} \in \ell^2$.
Since $\Pr(\hat{T}_n=T_\gamma)\to1$,
\begin{align*}
\langle \boldsymbol{u}, \boldsymbol{\mathcal{R}}_n \rangle
&= \langle \boldsymbol{u}_{T_\gamma}, \boldsymbol{\mathcal{R}}_{n T_\gamma} \rangle
-\langle \boldsymbol{u}, \boldsymbol{h}_n \rangle 
+ o_p(1) \\
&= \langle \boldsymbol{u}_{T_\gamma}, \boldsymbol{\mathcal{R}}_{n T_\gamma} \rangle
-\langle \boldsymbol{u}, \boldsymbol{h} \rangle + o_p(1).
\end{align*}
Hence it suffices to show the convergence in distribution of $\langle \boldsymbol{u}_{T_\gamma}, \boldsymbol{\mathcal{R}}_{n T_\gamma} \rangle$.
As in the proof of Theorem~\ref{thm:ell2_weak},
\begin{align*}
& (\boldsymbol{J}_{T_\gamma, T_\gamma})^{-1} (\hat{\boldsymbol{J}}_n)_{\hat{T}_n, \hat{T}_n} n^{1/2} (\tilde{\boldsymbol{\theta}}_n - \boldsymbol{\theta}_{0,n} )_{\hat{T}_n}1_{\{\hat{T}_n = T_\gamma\}} \\
&= n^{-1/2} \sum_{i=1}^n (\boldsymbol{J}_{T_\gamma, T_\gamma})^{-1} \boldsymbol{Z}_{i,n T_\gamma} \boldsymbol{Z}_{i,n T_{W,n}}^\top \boldsymbol{\theta}_{0,n T_{W,n}} 1_{\{\hat{T}_n = T_\gamma\}} 
\\&\quad 
+n^{-1/2} \sum_{i=1}^n (\boldsymbol{J}_{T_\gamma, T_\gamma})^{-1} \boldsymbol{Z}_{i,n T_\gamma} \epsilon_{i,n} 1_{\{\hat{T}_n = T_\gamma\}} .
\end{align*}
The first term on the right-hand side satisfies
\begin{align*}
&n^{-1/2} \sum_{i=1}^n (\boldsymbol{J}_{T_\gamma, T_\gamma})^{-1} \boldsymbol{Z}_{i,n T_\gamma} \boldsymbol{Z}_{i,n T_{W,n}}^\top \boldsymbol{\theta}_{0,n T_{W,n}} 1_{\{\hat{T}_n = T_\gamma\}} \\
&=n^{-1} 
\sum_{i=1}^n (\boldsymbol{J}_{T_\gamma, T_\gamma})^{-1} \boldsymbol{Z}_{i,n T_\gamma} \boldsymbol{Z}_{i,n T_{W,n}}^\top
\boldsymbol{h}_{n T_{W,n}} 1_{\{\hat{T}_n = T_\gamma\}}\\
&= (\boldsymbol{J}_{T_\gamma, T_\gamma})^{-1} (\hat{\boldsymbol{J}}_n)_{T_\gamma, T_{W,n}} \boldsymbol{h}_{n T_{W,n}}1_{\{\hat{T}_n = T_\gamma\}} \\
&= (\boldsymbol{J}_{T_\gamma, T_\gamma})^{-1} (\boldsymbol{J}_n)_{T_\gamma, T_{W,n}} \boldsymbol{h}_{n T_{W,n}} + o_p(1) \\
&= (\boldsymbol{J}_{T_\gamma, T_\gamma})^{-1} \boldsymbol{b}_h + o_p(1)
\end{align*}
and the second term is $O_p(1)$.
Hence,
\begin{align*}
&(\boldsymbol{J}_{T_\gamma, T_\gamma})^{-1} (\hat{\boldsymbol{J}}_n)_{\hat{T}_n, \hat{T}_n} n^{1/2} (\tilde{\boldsymbol{\theta}}_n - \boldsymbol{\theta}_{0,n} )_{\hat{T}_n} 1_{\{\hat{T}_n = T_\gamma\}} \\
&= n^{1/2} (\tilde{\boldsymbol{\theta}}_n - \boldsymbol{\theta}_{0,n} )_{T_\gamma} + o_p(1),
\end{align*}
which yields
\begin{align*}
& \boldsymbol{u}_{T_\gamma}^\top {n}^{1/2} (\tilde{\boldsymbol{\theta}}_n - \boldsymbol{\theta}_{0,n} )_{T_\gamma} \\
&=
n^{-1/2} \sum_{i=1}^n \boldsymbol{u}_{T_\gamma}^\top (\boldsymbol{J}_{T_\gamma, T_\gamma})^{-1} \boldsymbol{Z}_{i,n T_\gamma} \epsilon_{i,n} 
+ \boldsymbol{u}_{T_\gamma}^\top (\boldsymbol{J}_{T_\gamma, T_\gamma})^{-1} \boldsymbol{b}_h + o_p(1).
\end{align*}
We apply a Lyapunov-type central limit theorem for the first term on the right-hand side.
It holds that
\begin{align*}
& n^{-1} \sum_{i=1}^n \boldsymbol{u}_{T_\gamma}^\top  (\boldsymbol{J}_{T_\gamma, T_\gamma})^{-1} \boldsymbol{Z}_{i,n T_\gamma} \boldsymbol{Z}_{i,n T_\gamma}^\top (\boldsymbol{J}_{T_\gamma, T_\gamma})^{-1} \boldsymbol{u}_{T_\gamma} \epsilon_{i,n}^2 \\
& \to^p \sigma^2 \boldsymbol{u}_{T_\gamma}^\top (\boldsymbol{J}_{T_\gamma, T_\gamma})^{-1} \boldsymbol{u}_{T_\gamma}
\end{align*}
and, for any $\delta>0$,
\begin{align*}
&\sum_{i=1}^n \E[|
n^{-1/2} \boldsymbol{u}_{T_\gamma}^\top (\boldsymbol{J}_{T_\gamma, T_\gamma})^{-1} \boldsymbol{Z}_{i,n T_\gamma}\epsilon_{i,n} |^{2 + \delta}
] \\
&\leq 
n^{-\delta/2}\E[ |\boldsymbol{u}_{T_\gamma}^\top (\boldsymbol{J}_{T_\gamma, T_\gamma})^{-1} \boldsymbol{Z}_{1,n T_\gamma}|^{2+\delta}]\E[ |\epsilon_{1,n} |^{2+\delta}]
= o(1).
\end{align*}
Therefore, from Slutsky's theorem,
\[ \langle \boldsymbol{u}_{T_\gamma}, \boldsymbol{\mathcal{R}}_{n T_\gamma} \rangle 
\to^d \mathrm{N} (\boldsymbol{u}_{T_\gamma}^\top (\boldsymbol{J}_{T_\gamma, T_\gamma})^{-1} \boldsymbol{b}_h, \sigma^2 \boldsymbol{u}_{T_\gamma}^{\top} (\boldsymbol{J}_{T_\gamma, T_\gamma})^{-1} \boldsymbol{u}_{T_\gamma}) ,\]
which implies
\[
\langle \boldsymbol{u}, \boldsymbol{\mathcal{R}}_n \rangle
\to^d \mathrm{N} (\boldsymbol{u}_{T_\gamma}^\top ( \boldsymbol{J}_{T_\gamma, T_\gamma} )^{-1} \boldsymbol{b}_h -\langle \boldsymbol{u}, \boldsymbol{h} \rangle, \sigma^2 \boldsymbol{u}_{T_\gamma}^{\top} ( \boldsymbol{J}_{T_\gamma, T_\gamma} )^{-1} \boldsymbol{u}_{T_\gamma}).
\]
\end{itemize}
Theorem~1.8.4 of \cite{vanwellner2023} implies the conclusion.
\end{proof}

\section{Concluding remarks}\label{sec:CR}
For each $n\in\mathbb{N}$, we considered the high-dimensional linear model
\[ Y_{i,n} = \boldsymbol{\theta}_n^\top \boldsymbol{Z}_{i,n} +\epsilon_{i,n} = \sum_{j=1}^{p_n}\theta_{j,n} \boldsymbol{e}_j^\top \boldsymbol{Z}_{i,n} +\epsilon_{i,n}, \quad i\in[n], \]
assuming that $T_\gamma=\{j\in [p_n] : |\theta_{0j,n}|>\gamma\}$ is fixed for some $\gamma>0$.
We established weak convergence in $\ell^2$ for the post-selection estimation error and applied it to high-dimensional inference.
Unlike the fixed-support framework, our setting requires control of both the cumulative effect of diverging weak signals and the dependence among the corresponding covariates.

The present framework may be extended by replacing the canonical basis $\{\boldsymbol e_j\}_{j\in\mathbb N}$ with a general orthonormal basis.
Such an extension would allow regression coefficients that are not finitely supported in the canonical basis, under an appropriate notion of asymptotic finite-dimensionality with respect to a suitable orthonormal basis.

\section*{Funding}
This work was supported by Japan Society for the Promotion of Science KAKENHI Grant Numbers 25K07133 (KT) and 26K17032 (KF).

\bibliographystyle{chicago}
\bibliography{FT4-ref}

\end{document}